\documentclass[a4paper]{article}

\usepackage[utf8]{inputenc}
\usepackage[centertags]{amsmath}
\usepackage{amsthm}
\usepackage{amssymb}
\usepackage{amsfonts}

\RequirePackage{doi}
\usepackage{hyperref}
\hypersetup{colorlinks}
\usepackage{url}

\usepackage{algorithm}
\usepackage{algpseudocode}

\usepackage{multirow}
\usepackage{enumitem}

\usepackage{subcaption}
\usepackage{caption}
\usepackage{booktabs}

\usepackage{xcolor}
\usepackage{graphicx}
\definecolor{cb2red}{RGB}{228,26,28} 
\definecolor{cb2blue}{RGB}{55,126,184}

\usepackage{tikz}
\usepackage{pgfplots}
\pgfplotsset{compat=1.18}
\usepgfplotslibrary{groupplots}

\usepackage{fullpage}

\newcommand{\R}{\mathbb{R}}
\newcommand{\N}{\mathbb{N}}

\newcommand{\coloneqq}{:=}
\DeclareMathOperator*{\argmin}{\arg\min}
\DeclareMathOperator*{\minimize}{minimize}
\DeclareMathOperator{\stt}{such~that}
\DeclareMathOperator{\subjectto}{subject~to}
\DeclareMathOperator{\zeroset}{zer}
\DeclareMathOperator{\Fix}{Fix}
\DeclareMathOperator{\prox}{prox}

\DeclareMathOperator{\dom}{dom}
\newcommand{\indicator}{\iota}

\newcommand{\M}{\mathcal M}

\renewcommand{\H}{\ensuremath{\mathcal{H}}}

\newcommand{\ifc}{&\text{if }}
\newcommand{\gr}[1]{\mathrm{gra}({#1})}
\newcommand{\mcp}{\mathrm{mcp}}
\newcommand{\scad}{\mathrm{scad}}

\newcommand{\conv}{\mathrm{conv}}
\newcommand{\sgn}{\mathrm{sign}}
\newcommand{\reg}{\ensuremath{\mathcal{R}}}
\newcommand{\elcap}{\ifmmode \ell_{1}\text{--cap}\else $\ell_{1}$--cap\fi}

\newcommand{\tol}{\operatorname{tol}}
\newcommand{\A}{\mathbb{A}}
\newcommand{\B}{\mathbb{B}}
\newcommand{\NR}{\mathcal{N}}

\newtheorem{theorem}{Theorem}
	\newtheorem{lemma}[theorem]{Lemma}
	\newtheorem{proposition}[theorem]{Proposition}
	
	\newtheorem{corollary}[theorem]{Corollary}
	\theoremstyle{definition}
	\newtheorem{definition}{Definition}

\newtheorem{remark}{Remark}
\newtheorem{assumption}{Assumption}

\pgfplotsset{
  colormap={orange2red}{
    color(0cm)=(orange)
    color(1cm)=(red)
  }
}

\pgfplotsset{
  colormap={red2orange}{
    color(0cm)=(red)
    color(1cm)=(orange)
  }
}

\pgfplotsset{
  colormap={teal2green}{
    color(0cm)=(teal) 
    color(1cm)=(green) 
  }
}

\pgfplotsset{
  colormap={cyan2magenta}{
    color(0cm)=(cyan) 
    color(1cm)=(magenta) 
  }
}

\pgfplotsset{
  colormap={magenta2cyan}{
    color(0cm)=(magenta) 
    color(1cm)=(cyan) 
  }
}

\title{\bfseries Resolvent Moreau identities without monotonicity: theory and applications to Gabay duality, Douglas--Rachford and ADMM}
\author{Andrew Calcan\thanks{Centre for Optimisation and Decision Science, Curtin University, Australia}\and%
    Jordan Collard$^*$\and%
    Alberto De~Marchi\thanks{Department of Aerospace Engineering, University of the Bundeswehr Munich, Germany}\and%
    Scott B. Lindstrom$^*$}

\begin{document}
\maketitle

\newcommand{\mykeywords}{Duality, Splitting methods, Nonconvex optimization, Douglas--Rachford, ADMM.}
\newcommand{\mymsccodes}{49N15, 
65K10, 
90C46.
}

\begin{abstract}
    Duality is most often defined as a relationship between convex functions. If those functions are nonconvex, classical duality breaks down.
    Notwithstanding, we show that another kind of duality still exists, not between the functions themselves, but between the so-called resolvent operators used to solve associated problems.
    In fact, this duality-like relationship holds for any set-valued mapping, and is a generalization of the Moreau's identity.
    We use this duality to study existing operator schemes and to design new ones.
    In particular, we show that the duality-like relationship Daniel Gabay illuminated between the Douglas--Rachford splitting (DR) and the Alternating Direction Method of Multipliers (ADMM) extends to nonmonotone inclusion problems.
    We use this relationship to provide explicit counterexamples to the convergence of ADMM in several open cases, by studying the (easier to analyse) DR scheme.
    Motivated by our observations, we design a class of convergent resolvent homotopy schemes and use them to solve nonconvex-regularized least absolute deviations problems.
    This important problem class has received little attention in the literature, since the convex component of the objective does not enjoy strong convexity.
\end{abstract}

\medskip
\noindent\textbf{Keywords}
\mykeywords

\section{Introduction}

We study algorithms for locating a zero in the sum of two operators:
\begin{equation}
    \text{find }\lambda\in\H \;\stt\; \lambda\in \zeroset(\A+\B),
    \label{general_inclusion_problem}
\end{equation}
where $\H$ denotes a real Hilbert space, $\A,\B:\H\rightrightarrows \H$, and $\zeroset(S)\coloneqq \{y:0\in S(y)\}$. Problems of this form regularly appear across many disciplines, including statistics, engineering, physics and finance \cite{lindstrom2021survey}. Operator splitting methods are frequently used for \eqref{general_inclusion_problem} since they decompose the problem into smaller, computationally--tractable pieces.
Our considerations are motivated by the setting where \eqref{general_inclusion_problem} corresponds to the Lagrangian dual of a nonconvex problem, and specifically the nonconvex--regularized least absolute deviations problem: 
\begin{equation}
    \minimize_{x\in\R^n}\quad \big\| Ux-w \big\|_1 + \reg(x). 
    \label{rlad_problem}\tag{$\reg$LAD}
\end{equation}
Here $U\in\R^{m\times n}$ is a design matrix with corresponding response vector $w\in\R^m$.
Mapping $\reg:\R^n\rightarrow\R$ is a nonconvex regularization function whose inclusion in \eqref{rlad_problem} promotes sparse solutions.
Closely related to the classical LASSO problem \cite{tibshirani1996lasso}, \eqref{rlad_problem} is known to yield more robust and outlier insensitive statistical models. A common regularization choice is the $\ell_1$ norm, in which case the problem is convex and amenable to conic solvers. However, $\ell_1$ regularization is known to systematically shrink \emph{all} coefficients toward zero \cite{tibshirani2020penalties}. The smoothly clipped absolute deviation (SCAD) is widely accepted as the first nonconvex regularization function designed to mitigate the shrinkage bias induced by $\ell_1$ regularization \cite{fan2001scad}.
This idea was later expanded upon by the introduction of the minimax concave penalty (MCP) \cite{zhang2010mcp}.

The existing convergence guarantees for resolvent splitting approaches all require structure absent in \eqref{rlad_problem}.
Such structures are only realizable by approximating the problem.
This is commonly done by replacing the least absolute deviations term with a least squares term, which sacrifices robustness against outliers.
Relevant approaches include the methods proposed in \cite{dao2025dc,themelis2022linesearch}, which require one of the functions to enjoy Lipschitz continuous gradients.

Our approach is based on algorithmic duality.
Daniel Gabay first illuminated such a relationship between the Douglas--Rachford method (DR) and the alternating direction method of multipliers (ADMM) \cite{gabay1983duality}.
This relationship has led to improved understanding of both methods, with ADMM convergence results often given through analysis of DR \cite{eckstein2015duality,giselsson2017linear,yuan2012convergence}.
We refer the interested reader to \cite[Appendix 5]{lindstrom2021survey} for a detailed explanation of Gabay's original result.
Eckstein and Bertsekas later showed that DR, and therefore ADMM, is a special case of the proximal point algorithm \cite{bertsekas1992ppa}.
An accessible tutorial for analysing ADMM through its equivalence to DR is given in \cite{eckstein2015duality}.
This work is widely credited with clarifying and popularising the interpretation of ADMM as DR applied to the dual problem.
Another proof of the ADMM and DR relationship is given in the survey paper \cite{moursi2019equivalences}.
The methods' equivalence was later extended to certain nonconvex cases in \cite{themelis2020duality}.

\paragraph{Contributions}

Our main contributions are as follows.
We show in Theorem~\ref{thm:Moreaus_identity} that even with no assumptions on the operators or linear maps, the resolvent operators used by DR and ADMM enjoy a relationship that fully generalizes Moreau's identity.
Exploiting this relationship, we show in Theorem~\ref{thm:GMI_duality} that Gabay's duality-like relationship indeed characterizes DR and ADMM--type methods for all inclusion problems.
Consequently, Theorem~\ref{thm:GMI_duality} immediately allows practitioners to apply duality-based extrapolation schemes to nonconvex problems, including  dynamical systems approaches and the popular A$^3$DMM \cite{lindstrom2022computable,poon2019trajectory}. Based on Gabay's duality, we propose and analyse a provably convergent scheme for \eqref{general_inclusion_problem} under the assumption that the operators enjoy nonexpansive single--valued selections of their resolvents.

\paragraph{Outline}

This paper is organised as follows.
Section~\ref{sec:background} reviews basic concepts that will be used throughout.
In section~\ref{sec:Moreau_identity} we introduce and prove a generalization of Moreau's identity. We use this general expression in section~\ref{sec:duality} to prove a Gabay's duality between DR and a primal algorithm it induces in a general setting.
In section~\ref{sec:prim_dual_scheme} we introduce a duality--motivated scheme for addressing \eqref{general_inclusion_problem}, numerically validating this proposed scheme on \eqref{rlad_problem} with SCAD and MCP regularization in section~\ref{sec:experiment_results}.

\section{Background}\label{sec:background}
Here and throughout, $I$ denotes the identity mapping. We denote by $\Fix T = \{x : Tx=x\}$ the fixed point set of a set-valued operator $T:\H \rightrightarrows\H$; $T^{-1}$ is the set-valued `inverse' map that satisfies $y \in Tx \iff x \in T^{-1}y$. The graph of a potentially set-valued mapping $F$ is denoted by $\gr{F}:=\{(x,y):y\in F(x)\}$. For sets $U$ and $V$ we define $U+V:=\{u+v:u\in U,\; v\in V\}$. 

We now recall notions of generalized derivatives. 
\begin{definition}[Subdifferentials]
    Let $h:\H\rightarrow (-\infty,+\infty]$ be a proper and lower semicontinuous function. The Fr\'{e}chet subdifferential of $h$ at $x\in \H$ is given by
    $$
    \partial_{\mathrm{F}}h(x)\coloneqq \Big\{v\in \H :\underset{u\rightarrow x}{\lim\inf}\;\frac{h(u)-h(x)-\langle v,u-x\rangle}{\|u-x\|}\geq 0\Big\}.
    $$
    The Mordukhovich subdifferential of $h$ at $x\in \H$ is defined as
    $$
    \partial_{\mathrm{M}} h(x)\coloneqq \big\{x^*\in \H: \exists x_n\rightarrow x,\;x^*_n\in \partial_{\mathrm{F}}h(x_n)\text{ with } h(x_n)\rightarrow h(x),\;x_n^*\rightarrow x^*\big\}.
    $$
    Taking the convex hull of the Mordukhovich subdifferential yields the Clarke subdifferential, that is, $\partial h(x)\coloneqq \conv\{\partial_{\mathrm{M}}h(x)\}$. 
\end{definition}
Throughout we refer to the Clarke subdifferential simply as the subdifferential and denote it using $\partial$. 
\begin{theorem}[Clarke's chain rule {\cite[Theorem 2.6.6]{clarke1990analysis}}]\label{thm:Clarke}
    Let $x\in\H_1$ be arbitrary, $M:\H_1 \rightarrow \H$ be linear and $h: \H \to(-\infty,+\infty]$ be Lipschitz near $M(x)$.
    Then $h\circ M$ is Lipschitz near $x$ and it holds that
    $$
    \partial(h \circ M)(x) = M^* \partial h(Mx),
    $$
    where $M^*$ denotes the adjoint of $M$.
\end{theorem}

We denote by $\Gamma(X)$ the set of proper, lower semicontinuous and convex functions $h:X\rightarrow(-\infty,+\infty]$.
If $h\in\Gamma(\H)$, the Fr\'{e}chet, Mordukhovich and Clarke subdifferentials coincide with the classical subdifferential: $\{v: h(y)\geq h(x)+\langle v,y-x\rangle \;\forall y\in \H\}$.
Furthermore, $\partial h:\H\rightrightarrows \H$ is monotone: $\langle u - v,u' - v' \rangle \geq 0\; \forall (u,u'),(v,v')\in \gr{\partial h}$ and there exists no monotone operator whose graph strictly contains $\gr{\partial h}$.
The latter property implies that $\partial h$ is \emph{maximal} monotone.

\begin{definition}[Dual operator]\label{def:nonexpansive_dual_resolvents}
    Let $H:\H_1\rightrightarrows\H_1$, $M:\H_1\rightarrow \H$, and $d\in\H$.
    We refer to 
    $$
    D_H=(-M\circ H^{-1}\circ -M^*)(\cdot)-d
    $$
    as the \emph{dual operator} of $H$ with respect to $M$ and $d$.
\end{definition}
Herein we dedicate particular interest to the case of \eqref{general_inclusion_problem} where $\A$ and $\B$ are dual operators.
The relationship between this case and a problem in the primal space is described in the following proposition.
\begin{proposition}\label{prop:primal_dual_problems}
    Let $F:\H_x\rightrightarrows\H_x$ and $G:\H_z\rightrightarrows\H_z$.
    Let $A:\H_x\rightarrow \H$, $B:\H_z\rightarrow \H$ and $d_F,d_G\in\H$. Let 
    $$
    D_F=(-A\circ F^{-1}\circ -A^*)(\cdot)-d_F\quad \text{and}\quad D_G=(-B\circ G^{-1}\circ -B^*)(\cdot)-d_G
    $$
    be the dual operators of $F$ and $G$, respectively. For $\lambda\in\H$, the following are equivalent: 
    \begin{equation*}
    0\in D_F(\lambda)+D_G(\lambda);
    \label{dual_inclusion_problem}\tag{$\ensuremath{\mathcal{D}}$-inc}
    \end{equation*}
    there exists $x\in F^{-1}(-A^*\lambda)$ and $z\in G^{-1}(-B^*\lambda)$ such that
    \begin{equation}
    0\in
    \begin{bmatrix}
    F(x)+A^*\lambda\\
    G(z)+B^*\lambda\\
    Ax+Bz+d_F+d_G
    \end{bmatrix}.
    \label{primal_inclusion_problem}\tag{$\ensuremath{\mathcal{P}}$-inc}
    \end{equation}
\end{proposition}
\begin{proof}
    ($\Rightarrow$).
    Observe that $D_F(\lambda)=\{-Ax-d_F:x\in F^{-1}(-A^*\lambda)\}$ and $D_G(\lambda)=\{-Bz-d_G:z\in G^{-1}(-B^*\lambda)\}$.
    Suppose that \eqref{dual_inclusion_problem} holds.
    Then there exists $u_F\in D_F(\lambda)$ and $u_G\in D_G(\lambda)$ such that 
    $$
    0=u_F+u_G=(-Ax-d_F)+(-Bz-d_G)\iff 0= Ax+Bz+d_F+d_G.
    $$
    By how we have expressed $D_F(\lambda)$, we have that $-A^*\lambda\in F(x)$ as required.
    Similarly we also have $0\in G(z)+B^*\lambda$.

    ($\Leftarrow$).
    Suppose there exists $x\in F^{-1}(-A^*\lambda)$ and $z\in G^{-1}(-B^*\lambda)$.
    It follows that $-Ax-d_F\in D_F(\lambda)$ and $-Bz-d_G\in D_G(\lambda)$.
    The third condition of \eqref{primal_inclusion_problem} grants
    $$
    0\in Ax+Bz+d_F+d_F\iff 0\in (-Ax-d_F)+(-Bz-d_G).
    $$
    Since $-Ax-d_F\in D_F(\lambda)$ and $-Bz-d_G\in D_G(\lambda)$, we have $0\in D_F(\lambda)+D_G(\lambda)$. 
\end{proof}
An operator mapping diagram for the mappings in Proposition~\ref{prop:primal_dual_problems} is given in Figure~\ref{fig:operator_domains}, where we identify $\H_x$ and $\H_z$ with their duals via the Riesz representation theorem. Proposition~\ref{prop:primal_dual_problems} generalizes the relationship between \emph{primal} and \emph{dual} problems in the convex optimization setting.
This relationship is exploited and forms the basis for several popular and effective splitting methods \cite{calcan2024Lyapunov,pock2010cp,moursi2019equivalences,poon2019trajectory}.

The Legendre--Fenchel transform of $h:\H\to (-\infty,+\infty]$ is the mapping $h^*(u)\coloneqq \sup_x \langle u , x \rangle - h(u)$. If $F=\partial f$ and $G=\partial g$ for $f\in\Gamma(\H_x)$ and $g\in\Gamma(\H_z)$, and under mild conditions (c.f.~Theorem~\ref{thm:Clarke}), problem \eqref{dual_inclusion_problem} generalizes
\begin{equation}
    \minimize_{\lambda\in\H}\; f^*\left( -A^*\lambda\right) +g^*\left( -B^*\lambda\right) -\big\langle\lambda,d_F+d_G\big\rangle.
    \label{dual_minimisation_problem}\tag{$\ensuremath{\mathcal{D}}$-opt}
\end{equation}
This equivalence follows from the conjugacy relation $\partial h^*=(\partial h)^{-1}$ enjoyed by closed and convex $h$ \cite{bauschke2017monotoneoperators}.
Furthermore, the convexity of $f$ and $g$ ensures the well--known equivalence between \eqref{dual_minimisation_problem} and
\begin{equation}
    \minimize_{(x,z)\in\H_x\times\H_z}\quad f(x)+g(z) \quad\subjectto\quad Ax+Bz+d_F+d_G=0.
    \label{primal_optimisation_problem}\tag{$\ensuremath{\mathcal{P}}$-opt}
\end{equation}
We refer the reader to \cite{boyd2004convexoptimization} for a detailed discussion of Lagrangian duality. 

We say the tuple $(x^*,z^*)$ is a stationary point of \eqref{primal_optimisation_problem} if there exists $\lambda^*\in\H$ such that $(x^*,z^*,\lambda^*)$ satisfies \eqref{primal_inclusion_problem}.
If $F=\partial f$ and $G=\partial g$ for $f\in\Gamma(\H_x)$ and $g\in\Gamma(\H_z)$, and a standard constraint qualification holds e.g. $0\in \mathrm{ri}(A(\dom f)+B(\dom g)+d_F+d_G)$, then any stationary point $(x^*,z^*)$ is a solution to \eqref{primal_optimisation_problem}, and any solution to \eqref{primal_optimisation_problem} is necessarily a stationary point.
Problem duality in the convex optimisation and general operator settings is summarised in Table~\ref{table:problem_equivalences}.
\begin{table}[htb!]
    \centering%
    \begin{tabular}{c p{1.03cm} c}
    \textbf{Primal} && \textbf{Dual} \\
    \midrule
    \shortstack{$\displaystyle \minimize_{(x,z}\; f(x)+g(z)$ subject \\ to $Ax+Bz+d_F+d_G=0$}
    &
    \shortstack{$\overset{\text{(convex)}}{\Longleftrightarrow}$ \\ \color{white}I$^I$\color{black}}
    &
    \shortstack{$\displaystyle \minimize_{\lambda}\; f^*(-A^*\lambda)+$\\ $g^*(-B^*\lambda)-\langle \lambda,d_F+d_G\rangle$ \color{white}I$^I$\color{black}} \\
    &&\\
    \shortstack{find $(x,z,\lambda)$
    such that \\ $0\in
    \begin{bmatrix}
    F(x)+A^*\lambda\\[1pt]
    G(z)+B^*\lambda\\[1pt]
    Ax+Bz+d_F+d_G
    \end{bmatrix}$}
    &
    \shortstack{$\overset{\text{(always)}}{\Longleftrightarrow}$ \\ \color{white}I\color{black}}
    &
    \shortstack{find $\lambda$ such that \\ $0\in D_F(\lambda)+D_G(\lambda)$.}
\end{tabular}
    \caption{Primal and dual problems in the optimisation and operator settings.}%
    \label{table:problem_equivalences}%
\end{table}

We are primarily interested in schemes for \eqref{general_inclusion_problem} that utilise the resolvent operator:
\begin{definition}[Resolvent]
    Let $H:\H\rightrightarrows \H$.
    The \emph{resolvent} of $H$ is the, generically set--valued and possibly empty, operator $J_{H}\coloneqq (I+H)^{-1}$.
\end{definition}
If $H$ is maximal monotone, it is well known that its resolvent is nonempty, single--valued, and \emph{firmly nonexpansive} \cite{bauschke2017monotoneoperators}: $\|J_{H}x-J_{H}y\|^2\leq \langle x-y,J_{ H}x-J_{H}y\rangle$ for all $x,y\in\dom(H)$. The latter property is strictly stronger than nonexpansivity: $\| J_Hx-J_Hy\| \leq \| x-y \|$ for all $x,y\in\dom(H)$.
\begin{definition}[Selection operator]
    The mapping $S_H:\H\to \H$ is a \emph{selection operator} for $J_H:\H\rightrightarrows \H$ if it satisfies the following conditions:
    \begin{enumerate}[label=(\roman*)]
        \item $\left( J_H(x) = \emptyset\right)  \quad\implies\quad \left( S_H(x) = \emptyset\right) $; and
        \item $\left( J_H(x) \neq \emptyset\right)  \quad\implies\quad \left( S_H(x) \in J_H(x)\right) $.
    \end{enumerate}
\end{definition}
In practice, selection operators are often utilized in cases where resolvents are set valued.
We note that the selection operators considered herein are deterministic in the sense that the same element of $J_H(y)$ is always chosen.
If $h\in\Gamma(\H)$ and $\gamma>0$, $J_{\gamma (\partial h)}$ agrees with the \emph{proximal operator} of $h$:
$$
J_{\gamma (\partial h)}(x)=\prox_{\gamma h}(x)\coloneqq \argmin_{y\in\H} \left\{ h(y)+\frac{1}{2\gamma}\big\|y-x\|^2 \right\}.
$$
If $h$ is nonconvex, its proximal operator can be used as a selection operator for $J_{\gamma (\partial h)}$ as follows. 
\begin{proposition}\label{prop:prox_is_selection}
    Suppose $h$ is locally Lipschitz near every point of its domain.
    Let $\gamma>0$.
    Then a selection of $\prox_{\gamma h}$ is a selection mapping for the resolvent of the Clarke subdifferential $J_{\gamma (\partial h)}$.
\end{proposition}
\begin{proof}
    Because $h$ is locally Lipschitz, so also is the function $u \mapsto h(u) + \frac{1}{2\gamma}\|u-x\|^2$, wherefore any minimizer must satisfy the Clarke necessary optimality condition $0 \in \partial(h + \frac1{2\gamma}\|\cdot-x\|^2)(u) = \partial h(u) + \gamma^{-1}(u-x)\implies u \in (I+\gamma \partial h)^{-1}(x)$.
\end{proof}

\section{A Moreau--like identity for resolvents of set-valued mappings}\label{sec:Moreau_identity}

\begin{definition}[Generalized resolvent]\label{def:generalised_resolvents}
    Let $H:\H_1\rightrightarrows \H_1$ and $U:\H_1\rightarrow \H$.
    The \emph{generalized resolvent} of $H$ is the mapping 
    $$
    J_{H}^U \coloneqq (U^* U + H)^{-1}.
    $$
\end{definition}
In the following theorem we establish a connection between the dual resolvent of an operator $\smash{J_{\gamma D_H}}$ and the generalized resolvent $\smash{J^U_{\gamma^{-1}H}}$.
This is, to the author's knowledge, the most general expression analogous to Moreau's identity in the operator setting.
This expression may be used to develop dualities between resolvent methods, as we will see.
\begin{theorem}[Generalized Moreau identity]\label{thm:Moreaus_identity}
    Let $H:\H_1\rightrightarrows\H_1$, $M:\H_1\rightarrow\H$ be linear and $d\in\H$.
    Let $\gamma >0$ and $D_H=(-M\circ H^{-1}\circ-M^*)(\cdot)- d$.
    Then
    \begin{equation}
    J_{\gamma D_H}(u) = u + \gamma d + \gamma M J^M_{\gamma^{-1} H}(-M^*(\gamma^{-1}u +d))
    \qquad\forall u\in\H.
    \label{GMI}\tag{GMI}
    \end{equation}
    This identity holds even when the operators $J_{\gamma D_H}$ and $J^M_{\gamma^{-1} H}$ do not have full domain; i.e. may have empty images.
    Furthermore, if $\lambda\in J_{\gamma D_H}(u)$, there exists $p,q\in \H$ such that the following hold:
    \begin{enumerate}[label=(\roman*)]
        \item $p\in D_H(\lambda)$ and $u=\gamma p+\lambda$; \label{GMI_item1}
        \item $q\in H^{-1}(-M^*\lambda)$ and $p=-Mq-d$; \label{GMI_item2}
        \item $q\in J^M_{\gamma^{-1} H}(-M^*(\gamma^{-1}u+d))$. \label{GMI_item3}
    \end{enumerate}
\end{theorem}
\begin{proof}
    The proof uses straightforward manipulations of set--valued operators and linear maps.
    To show the result, it suffices to fix any $u \in \H$ and show the set equality:
    $$
    (I+\gamma D_H)^{-1}(u) = u + \gamma d + \gamma M(M^*M+\gamma^{-1}H)^{-1}(-M^*(\gamma^{-1} u +d)).
    $$
    This is what we will do.
    
    ($\subseteq$).
    The result is vacuous if $J_{\gamma D_H}(u) = \emptyset$.
    Let $\lambda \in J_{\gamma D_H}(u)$, which implies $u=\gamma p+\lambda$ for some $p \in D_H(\lambda)$.
    This shows \ref{GMI_item1}.
    Notice $p\in D_H(\lambda)$ implies the existence of $q \in H^{-1}(-M^*\lambda)$ such that $p=-Mq-d$.
    This shows \ref{GMI_item2}.
    It follows that
    \begin{align*}
        q \in H^{-1}(-M^*\lambda) \quad\iff& \quad Hq \ni-M^* \lambda \\
        \iff& \quad (M^*M+\gamma^{-1}H)q \ni M^*Mq -\gamma^{-1}M^* \lambda\\
        \iff& \quad q \in (M^*M+\gamma^{-1}H)^{-1}(-M^*(\gamma^{-1}\lambda-Mq))\\
        \iff& \quad q\in (M^*M+\gamma^{-1}H)^{-1}(-M^*(\gamma^{-1}\lambda+p+d))\\
        \iff& \quad q\in (M^*M+\gamma^{-1}H)^{-1}(-M^*(\gamma^{-1}\lambda+\gamma^{-1}(u-\lambda)+d))\\
        \iff& \quad q\in (M^*M+\gamma^{-1}H)^{-1}(-M^*(\gamma^{-1}u+d)),
    \end{align*}
    with the final line showing \ref{GMI_item3}. Altogether we have
    \begin{align*}
        \lambda &= u - \gamma p = u-\gamma(-Mq-d) \in u + \gamma d+\gamma M(M^*M+\gamma^{-1}H)^{-1}(-M^*(\gamma^{-1}u+d)).
    \end{align*}
    This completes the proof of ($\subseteq$).

    ($\supseteq$).
    If $u+\gamma d+\gamma M(M^*M+\gamma^{-1}H)^{-1}(-M^*(\gamma^{-1}u+d))=\emptyset$, then the result is vacuous.
    Otherwise, there exists $q \in (M^*M+\gamma^{-1}H)^{-1}(-M^*(\gamma^{-1}u+d))$. Thus:
    \begin{align*}
        q \in (M^*M+\gamma^{-1}&H)^{-1}(-M^*(\gamma^{-1}u+d)) \\
        \iff& \quad  \gamma^{-1}Hq + M^*Mq \ni -M^*(\gamma^{-1}u+d)\\
        \iff&\quad Hq\ni -M^*(u+\gamma d)-\gamma M^*Mq\\
        \iff&\quad  q \in H^{-1}(-M^*(u+\gamma d+\gamma Mq))\\
        \iff&\quad -\gamma Mq-\gamma d \in -\gamma M(H^{-1}(-M^*(u+\gamma d+\gamma Mq)))-\gamma d\\
        \iff& \quad  -\gamma Mq-\gamma d\in \gamma D_H(u+\gamma d+\gamma Mq)\\
        \iff& \quad -\gamma Mq-\gamma d + (u+\gamma d+\gamma Mq) \in (I+\gamma D_H)(u+\gamma d+\gamma Mq)\\
        \iff&\quad u\in (I+\gamma D_H)(u+\gamma d+\gamma Mq)\\
        \iff& \quad J_{\gamma D_H}(u) \ni u+\gamma d+\gamma Mq.
    \end{align*}
    This establishes the direction $(\supseteq)$, concluding the proof.
\end{proof}
\begin{remark}
    Specifying $M=-I$, $d=0$ and $H = \partial h$ for $h\in\Gamma(\H)$, one sees that \eqref{GMI} is a generalisation of Moreau's proximal identity \cite{moreau1965proximal}: 
    \begin{equation}
    \prox_{\gamma h^*}(u) =u - \gamma \prox_{\gamma^{-1}h}(\gamma^{-1}u).
    \label{Moreau_identity}
    \end{equation}
\end{remark}

The following corollary is an immediate consequence of Theorem~\ref{thm:Moreaus_identity} and shows that \eqref{GMI} necessarily holds for selections.
\begin{corollary}\label{selection_GMI_corollary}
    If $S^M_{\gamma^{-1}H}$ is a selection operator for $J^M_{\gamma^{-1}H}$, then there exists a selection mapping $S_{\gamma D_H}$ for $J_{\gamma D_H}$ that satisfies
    \begin{equation}
        S_{\gamma D_H}(u) = u + \gamma d+ \gamma M S_{\gamma^{-1}H}^M(-M^*(\gamma^{-1}u+ d))
        \qquad\forall u\in\H.
    \label{GMI_selections}
    \end{equation}
\end{corollary}
\begin{proof}
    The proof is that of Theorem~\ref{thm:Moreaus_identity}, with equalities replacing inclusions.
\end{proof}

\section{A general Gabay duality for DR}\label{sec:duality}

DR is employed for problems of the form \eqref{general_inclusion_problem} under the typical assumption that $\A,\B:\H\rightrightarrows\H$ are maximal monotone \cite{mercier1979dr}.
The method utilises the reflected resolvent $R_{\gamma H}\coloneqq 2J_{\gamma H}-I$ of an operator $H$.
\begin{theorem}[DR] \label{thm:DR}
    Let $\gamma>0$ and $\A,\B:\H\rightrightarrows\H$.
    Suppose that $R_{\gamma\A}$ and $R_{\gamma\B}$ are nonexpansive.
    DR generates sequence $(y^k)$ according to 
    \begin{equation}
    y^{k+1}\in T_{\A,\B}(y^k)
    ,\quad\text{where}\quad
    T_{\A,\B}\coloneqq \frac{1}{2}\left( I+ R_{\gamma \A}\circ R_{\gamma \B}\right).
    \label{DRM}
    \end{equation}
    This sequence converges weakly to some $u \in \Fix T_{\A,\B}$ (assuming $\Fix T_{\A,\B}\neq \emptyset$) such that $0\in \A(J_{\gamma \B}u)+\B(J_{\gamma \B}u)$.
\end{theorem}
\begin{proof}
    The composition $R_{\gamma\A}\circ R_{\gamma\B}$ is nonexpansive by the nonexpansiveness of $R_{\gamma\A}$ and $R_{\gamma\B}$.
    It follows that $T_{\A,\B}$ is $1/2$-averaged and firmly nonexpansive.
    Since (by assumption) $\Fix T_{\A,\B}\neq\emptyset$, the sequence $(y^k)_k$ converges weakly to $u\in\Fix T_{\A,\B}$ (see, for example, \cite[Example~5.18]{bauschke2017monotoneoperators}).
    It is well known and straightforward to verify that $u\in\Fix T_{\A,\B}$ is such that $J_{\gamma \B}u\in\zeroset(\A+\B)$ \cite{mercier1979dr}.
\end{proof}

The nonexpansiveness of $R_{\gamma\A}$ and $R_{\gamma\B}$ is automatic when $\A$ and $\B$ are maximal monotone. A combination of theoretical and practical results demonstrate DR's aptitude for non maximal monotone problems where the reflections lack nonexpansiveness \cite{lindstrom2021survey}.
These include the sphere and line \cite{aragon2013global}, graph coloring \cite{aragon2018solving}, and matrix completion problems \cite{aragon2014douglas,borwein2012reflection}. DR convergence analyses in the nonconvex setting often rely on transversality or metric subregularity. These properties were used to show linear convergence for specific problems in \cite{dao2019linear, hesse2013nonconvex, phan2016linear}.

It is common in practice to attempt to solve \eqref{primal_optimisation_problem} by approaching \eqref{dual_inclusion_problem} with DR.
This is less common if \eqref{primal_optimisation_problem} is nonconvex, since solution equivalence is no longer guaranteed.
It was shown in \cite{gabay1983duality} that if $f\in\Gamma(\H_x)$, $g\in\Gamma(\H_z)$, $B=-I$ and $d_F=d_G=0$, applying DR to \eqref{dual_minimisation_problem} is equivalent to applying ADMM to \eqref{primal_optimisation_problem}.

ADMM is a popular splitting method that addresses \eqref{primal_optimisation_problem} according to
\begin{subequations}
    \label{ADMM_updates}
    \begin{align}
    x^{k+1}&\in \argmin_{x\in\H_x} \left\{ f(x)+\frac{\rho}{2}\big\|Ax+Bz^k+\rho^{-1}\lambda^k+d_F+d_G\big\|^2 \right\}, \label{ADMM_x_update}
    \\
    z^{k+1}&\in \argmin_{z\in\H_z} \left\{ g(z)+\frac{\rho}{2}\Big\|Ax^{k+1}+Bz+\rho^{-1}\lambda^k+d_F+d_G\big\|^2 \right\}, \label{ADMM_z_update}
    \\
    \lambda^{k+1}&=\lambda^k+\rho\left( Ax^{k+1}+Bz^{k+1}+d_F+d_G\right)  , \label{ADMM_dual_update}
    \end{align}
\end{subequations}
for some given $\rho>0$.

In the maximal monotone setting, the DR sequence $(y^k)$ is known to be linked to the ADMM sequence $(x^k,z^k,\lambda^k)$ by the expressions in Table~\ref{tab:gabay_identities}, see \cite{lindstrom2022computable}. We find that these identities, which we call the \emph{Gabay identities}, also hold in general. We extend Gabay's result in Theorem~\ref{thm:GMI_duality}, wherein we show that DR applied to \eqref{dual_inclusion_problem} induces a primal algorithm akin to ADMM for \eqref{primal_inclusion_problem}. Our result is a generalization in the sense that we account for sequences of set--valued operators and the constraint operators are only assumed to be linear.

\begin{table}
    \small
    \centering
    \begin{tabular}{c}
    \textbf{ADMM} \\
    \hline
    $\rho Ax^{k+1}\in J_{\rho D_{F^{k+1}}} (R^k_{\rho D_{G^k}}y^k)-R^k_{\rho D_{G^k}}y^k-\rho d_{F^{k+1}}$ \\
    $\rho Bz^k=\lambda^k-y^k-\rho d_{G^k}$ \\
    $\lambda^k\in J_{\rho D_{G^k}}y^k$ \\
    \textbf{DR} \\
    \hline
    $y^k=\lambda^{k-1}+\rho(Ax^k+d_{F^{k+1}})$\\
    $R^k_{\rho D_{G^k}}y^k=\lambda^k+\rho(Bz^k+d_{G^k})$\\
    $J_{\rho D_{F^{k+1}}} (R^k_{\rho D_{G^k}}y^k)\ni\lambda^k+\rho(Ax^{k+1}+Bz^k+d_{F^{k+1}}+d_{G^k})$
    \end{tabular}
    \caption{Expressions connecting ADMM and DR iterates.}
    \label{tab:gabay_identities}
\end{table}

\begin{theorem}\label{thm:GMI_duality}
    Let $(F^k)$ and $(G^k)$ be sequences of operators with $F^{k}:\H_x\rightrightarrows\H_x$ and $G^k:\H_z\rightrightarrows\H_z$.
    Let $A:\H_x\rightarrow \H$ and $B:\H_z\rightarrow\H$ be linear and $d_{F^{k}},d_{G^k} \in\H$.
    Let $D_{F^k}$ denote the dual operator of $F^k$ with respect to $A$ and $d_{F^k}$ and $D_{G^k}$ denote the dual operator of $G^k$ with respect to $B$ and $d_{G^k}$.
    The DR sequence given by
    $$
    y^{k+1}\in  \frac{1}{2}\left(I+ R_{\rho D_{F^{k+1}}}\circ R_{\rho D_{G^k}}\right)y^k
    $$
    induces a sequence $(x^k,z^k,\lambda^k)$ satisfying
    \begin{subequations}
    \label{primal_alg_updates}
    \begin{align}
            x^{k+1}&\in J^A_{\rho^{-1}F^{k+1}}\left( -A^*(Bz^k+\rho^{-1}\lambda^k+d_{F^{k+1}}+d_{G^k})\right) , \label{primal_alg_x_update}\\
            z^{k+1}&\in J^B_{\rho^{-1}G^k}\left( -B^*(Ax^{k+1}+\rho^{-1}\lambda^k+d_{F^{k+1}}+d_{G^k})\right) , \label{primal_alg_z_update}\\
            \lambda^{k+1}&=\lambda^k+\rho(Ax^{k+1}+Bz^{k+1}+d_{F^{k+1}}+d_{G^k}).\label{primal_alg_dual_update}
        \end{align}
    \end{subequations}
    Furthermore, the sequences $(x^k,z^k,\lambda^k)_k$ and $(y^k)_k$ are connected by the Gabay identities of Table~\ref{tab:gabay_identities}.
\end{theorem}
\begin{proof}
    Choose $\lambda^{-1}$, $y^0$ and $x^0$ such that $\lambda^{-1}=y^0-\rho (Ax^0+d_{F^{k+1}})$, from which it follows that $\lambda^{k-1}=y^k-\rho(Ax^k+d_{F^{k+1}})$.
    Let $\lambda^k\in J_{\rho D_{G^k}}y^k$.
    We have by Theorem~\ref{thm:Moreaus_identity} that there exists $v^k\in\H$ and $z^k\in \H_z$ such that the following hold:
    \begin{enumerate}[label=(\roman*)]
        \item $v^k\in D_{G^k}(\lambda^k)$ and $y^k=\rho v^k+\lambda^k$;
        \item $z^k\in (G^k)^{-1} (-B^*\lambda^k)$ and $v^k=-Bz^k-d_{G^k}$;\label{thm:duality_2}
        \item \label{thm:duality_3}$z^k\in J^B_{\rho^{-1} G^k}\left( -B^*(\rho^{-1}y^k+ d_{G^k}) \right) $.
    \end{enumerate}
    Then, let $R_{D_{G^k}}y^k= \lambda^k+\rho(Bz^k+d_{G^k})$ and $\mu^k\in J_{\rho D_{F^k}}(R_{D_{G^k}}y^k)=J_{\rho D_{F^k}}(\lambda^k-\rho v^k)$.
    Owing to Theorem~\ref{thm:Moreaus_identity}, there exists $w^k\in\H$ and $x^{k+1}\in \H_x$ such that
    \begin{enumerate}[start=4,label=(\roman*)]
        \item $w^k\in D_{F^{k+1}}(\mu^k)$ and $\lambda^k-\rho v^k=\rho w^k+\mu^k$;
        \item $x^{k+1}\in (F^{k+1})^{-1}(-A^*\mu^k)$ and $w^k=-Ax^{k+1}-d_{F^{k+1}}$;
        \item \label{thm:duality_6}$x^{k+1}\in J^A_{\rho^{-1}F^{k+1}}(-A^*(\rho^{-1}(\lambda^k-\rho v^k)+ d_{F^{k+1}}))$.
    \end{enumerate}
    It follows from \ref{thm:duality_2} and \ref{thm:duality_6} that
    \begin{multline*}
        \rho^{-1}F^{k+1}(x^{k+1})+ A^*Ax^{k+1}
        \ni
        -A^*(\rho^{-1}(\lambda^k-\rho v^k)+ d_{F^{k+1}})
        \\=
        -A^*(\rho^{-1}\lambda^k+Bz^k+d_{F^{k+1}}+d_{G^k}),
    \end{multline*}
    from which we have
    \begin{equation}
        0\in F^{k+1}(x^{k+1})+\rho A^*(Ax^{k+1}+Bz^k+\rho^{-1}\lambda^k+d_{F^{k+1}}+d_{G^k}).
        \label{x_zero_condition}
    \end{equation}
    We can equivalently express \eqref{x_zero_condition} as $x^{k+1}\in J^A_{\rho^{-1}F^{k+1}}(-A^*(Bz^k+\rho^{-1}\lambda^k+d_{F^{k+1}}+d_{G^k}))$.
    Combining \ref{thm:duality_3} with the fact that $\lambda^{k-1}=y^k-\rho(Ax^k+d_{F^{k+1}})$ yields
    \begin{equation*}
        \rho^{-1}G^k(z^k)+B^*Bz^k \ni -B^*(\rho^{-1}y^k+d_{G^k})=-B^*(\rho^{-1}\lambda^{k-1}+Ax^k+d_{F^{k+1}}+d_{G^k}).
    \end{equation*}
    It follows that $z^k\in J^B_{\rho^{-1}G^k}(-B^*(Ax^k+\rho^{-1}\lambda^{k-1}+d_{F^{k+1}}+d_{G^k}))$.
    Now considering the dual update, we have that $\lambda^k-\rho(Bz^k+d_{G^k})=\lambda^k+\rho v^k=y^k= \lambda^{k-1}+\rho(Ax^k+d_{F^{k+1}})$, which implies $\lambda^k=\lambda^{k-1}+\rho(Ax^k+Bz^k+d_{F^{k+1}}+d_{G^k})$.
    Collecting these results, we see that DR induces primal algorithm \eqref{primal_alg_updates}.
    The one-to-one connection via the Gabay identities concludes the proof.
\end{proof}
The connection between ADMM and DR variables for a single iteration is visualized in Figure~\ref{eqn:DR-ADMM-duality-proof}. The DR updates and intermediate steps are denoted using filled and unfilled circles, respectively.
\begin{remark}
    It is straightforward to verify that if $F=\partial f$ and $G=\partial g$ for $f\in\Gamma(\H_x)$ and $g\in\Gamma(\H_z)$, then \eqref{primal_alg_updates} is equivalent to \eqref{ADMM_updates} and hence DR is equivalent to ADMM.
    We show in Theorem~\ref{thm:primal_convergence}, for a slightly more general case, that if DR converges for \eqref{dual_inclusion_problem} then the sequences of \eqref{primal_alg_updates} converge to a solution of \eqref{primal_inclusion_problem}.
\end{remark}

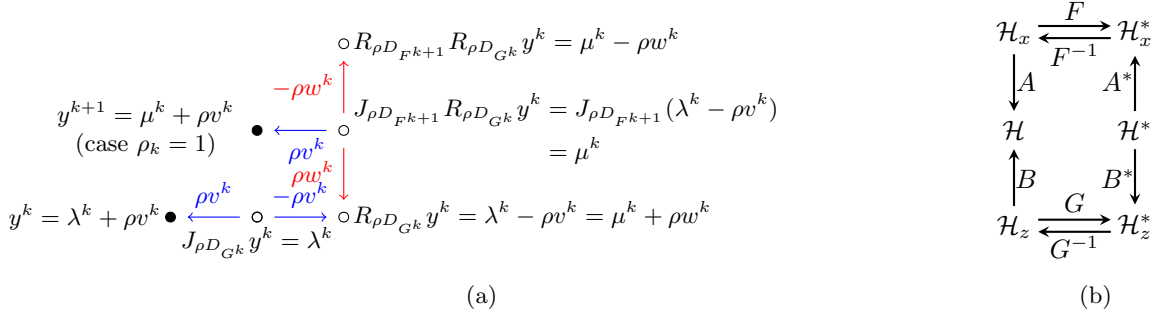
\begin{figure}[htb]
\centering
\begin{subfigure}{0.8\textwidth}
\begin{tikzpicture}[scale=0.81]

\begin{scope}[rotate=-45]  

\draw[fill] (0, 0) circle (.08cm);

\draw (1, 1) circle (.08cm);
\draw (2, 2) circle (.08cm);
\draw (1, 3) circle (.08cm);
\draw (0, 4) circle (.08cm);
\draw (1, 1) circle (.08cm);
\draw[fill] (0, 2) circle (.08cm);

\draw[<-,blue] (0+.2,0+.2) -- (1-.2, 1-.2);
\draw[->,blue] (1+.2,1+.2) -- (2-.2, 2-.2);
\draw[<-, red] (2-.2,2+.2) -- (1+.2, 3-.2);
\draw[->, red] (1-.2,3+.2) -- (0+.2, 4-.2);
\draw[->,blue] (1-.2,3-.2) -- (0+.2, 2+.2);

\node[font=\small,above,blue] at (0.5,0.5) {$\rho v^k$};
\node[font=\small,above,blue] at (1.5,1.5) {$-\rho v^k$};
\node[font=\small,left,red] at (1.5,2.5) {$\rho w^k$};
\node[font=\small,left,red] at (0.5,3.5) {$-\rho w^k$};
\node[font=\small,below,blue] at (0.55,2.6) {$\rho v^k$};

\node[font=\small,left] at (0,0) {$y^k = \lambda^k+\rho v^k$};
\node[font=\small,below] at (1,1) {$J_{\rho D_{G^k}}y^k = \lambda^k$};
\node[font=\small,right] at (2,2) {$R_{\rho D_{G^k}}y^k = \lambda^k-\rho v^k = \mu^k + \rho w^k$};
\node[font=\small,right] at (1,3) {$\begin{aligned}J_{\rho D_{F^{k+1}}} R_{\rho D_{G^k}}y^k &= J_{\rho D_{F^{k+1}}}(\lambda^k-\rho v^k) \\&= \mu^k\end{aligned}$};
\node[font=\small,right] at (0,4) {$R_{\rho D_{F^{k+1}}}R_{\rho D_{G^k}}y^k = \mu^k - \rho w^k$};
\node[font=\small,left] at (0,2) {$\begin{array}{c }y^{k+1} = \mu^k + \rho v^k\\
\text{(case $\rho_k =1$)}\end{array}$};
\end{scope}
\end{tikzpicture}
\caption{}
\label{eqn:DR-ADMM-duality-proof}
\end{subfigure}
\hfill
\begin{subfigure}{0.18\textwidth}
\begin{tikzpicture}[>=stealth, thick,scale=0.32]
\node (H) at (-2.5,0) {$\H$};
\node (Hstar) at (2.5,0) {$\H^*$};

\node (Hz)     at (-2.5,-4) {$\H_z$};
\node (Hzstar)     at ( 2.5,-4) {$\H_z^*$};
\node (Hx) at (-2.5, 4) {$\H_x$};
\node (Hxstar) at (2.5,4) {$\H_x^{*}$};

\node at (1.8,2) {$A^*$};
\node at (-2,2) {$A$};
\node at (1.8,-2) {$B^*$};
\node at (-2,-2) {$B$};

\draw[->] (Hx) to (H);
\draw[->] (-1.5,4.3) to node[midway,below,yshift=-3pt] {$F^{-1}$} (1.5,4.3);

\draw[->] (1.5,3.8) to node[midway,above,yshift=3pt] {$F$} (-1.5,3.8);

\draw[->] (Hz) to (H);
\draw[->] (-1.5,-3.7) to node[midway,above,yshift=-1pt] {$G$} (1.5,-3.7);

\draw[->] (1.5,-4.2) to node[midway,below,yshift=2pt] {$G^{-1}$} (-1.5,-4.2);

\draw[->] (Hstar) to (Hxstar);
\draw[->] (Hstar) to (Hzstar);

\end{tikzpicture}
\caption{}
\label{fig:operator_domains}
\end{subfigure}
\caption{(\subref{eqn:DR-ADMM-duality-proof}) DR expressed using primal variables. (\subref{fig:operator_domains}) An operator mapping diagram for the single--valued case of $F$ and $G$ and $A$ and $B$ and their adjoints.}
\label{fig:duality_and_domains}
\end{figure}

Theorem~\ref{thm:GMI_duality} demonstrates---provided the updates are defined---Gabay's duality-like relationship always holds, regardless of the properties of $F,G$ or the linear maps $A$ and $B$.
Gabay's identities reveal that at each iteration $k$, a primal algorithm tuple $(x^k,z^k,\lambda^k)$ is computed at no added cost when DR is applied to \eqref{dual_inclusion_problem}.
This relationship enables methods that extrapolate on the sequence $(y^k)$, such as those in \cite{calcan2024Lyapunov,lindstrom2022computable,poon2019trajectory}.
An immediate application of Theorem~\ref{thm:GMI_duality} is that such extrapolation schemes may now be applied to nonmonotone inclusion problems.
These identities also show that ADMM may be interpreted as being \emph{induced} by DR, or vice--versa.

\begin{remark}
    Theorem~\ref{thm:GMI_duality} establishes the connection between DR and a method that generalizes ADMM. Algorithm \eqref{primal_alg_updates} generalizes ADMM in the sense that updates \eqref{primal_alg_x_update} and \eqref{primal_alg_z_update} correspond to locating stationary points of \eqref{ADMM_x_update} and \eqref{ADMM_z_update} if $F=\partial f$ and $G=\partial g$ for proper and lower semicontinuous $f$ and $g$.
    Additionally, $A$ and $B$ are general, not necessarily injective, linear operators.
    This allowance is important in stochastic programming, where the operators in question (often described implicitly) are generally not injective but consist of blocks of identity and zero matrices.
\end{remark}

\subsection{Gabay duality in practice}\label{ex:basis_pursuit}

We now illustrate the connection between DR and its induced primal algorithm with an example.
This example also explicitly shows that DR may fail to converge for $1$--weakly convex problems.
Let us consider the basis pursuit problem, which has applications in signal recovery; see, for example, \cite{calcan2024Lyapunov}.
The formulation is given by
\begin{equation}
    \minimize_{x\in\R^n}\quad \reg_{\mcp}(x) \quad\subjectto\quad Ux=w,
    \label{problem:basis_pursuit}\tag{$\ensuremath{\mathcal{P}}$-BP}
\end{equation}
where $U\in\R^{m\times n}$ and $w\in\R^m$ are such that the linear system $Ux=w$ is consistent and underdetermined.
We have $\reg_{\mcp}(x)\coloneqq \sum_{i=1}^n\phi_{\mcp}(x_i)$, where $\phi_{\mcp}\colon\R\to\R$ denotes the MCP function with parameter $\beta>0$ and strength $\lambda>0$: 
\begin{equation*}
    \phi_{\mcp}(x)\coloneqq
    \begin{cases}
        \lambda |x| - \frac{x^2}{2\beta} \ifc |x| \leq \beta \lambda , \\
        \frac{\beta \lambda^2}{2} \ifc |x| > \beta \lambda.
    \end{cases}
\end{equation*} 
In this example we set $\beta=1$ and $\lambda=2$.
This choice of regularizer encourages sparse solutions to the linear system.
We consider basis pursuit not because a nonconvex regularizer would afford better objective value, but because this problem is easy to visualize and study in low dimensions to learn about behaviour.

The dual inclusion corresponding to \eqref{problem:basis_pursuit} is to find $\lambda\in\R^n$ such that 
\begin{equation}
    0\in\underbrace{-\partial \indicator_{\M}^*(-\lambda)}_{=D_F(\lambda)} + \underbrace{(\partial\reg_{\mcp})^{-1}(\lambda)}_{=D_G(\lambda)}.
    \label{problem:basis_pursuit_dual}\tag{$\ensuremath{\mathcal{D}}$-BP}
\end{equation}
Here $\M = \{ x\in \R^n \colon Ux=w\}$ and $\indicator^*_{\M}(y) = \sup_{x\in\M}\langle y,x\rangle$.
We apply DR to \eqref{problem:basis_pursuit_dual} using the resolvents 
\begin{multline*}
    J_{\gamma D_F}(y)=U^{\top}(UU^{\top})^{-1}(Uy+\gamma w)
    \quad\text{and}\\
    J_{\gamma D_G}(y)=(S_{\gamma(\partial \phi_\mcp)^{-1}}(y_1),\ldots,S_{\gamma(\partial \phi_\mcp)^{-1}}(y_n)).
\end{multline*}
An expression for selection $\smash{S_{\gamma(\partial \phi_\mcp)^{-1}}}$ is
\begin{equation}
    S_{\gamma(\partial \phi_\mcp)^{-1}}(x)=
    \begin{cases}
        x&\text{if } |x|\leq\lambda\\
        \frac{x-\sgn(x)\beta\gamma\lambda}{1-\beta\gamma}
        &\text{if }|x|\in({\lambda},\beta\gamma\lambda)\\0
        &\text{if } |x|\geq \beta\gamma\lambda.
    \end{cases}
    \label{mcp_dual_resolvent}
\end{equation}
We refer the reader to Appendix~\ref{app:mcp_scad_derivations} for a derivation of \eqref{mcp_dual_resolvent}.

The relationship between DR and its primal algorithm is visualized in Figure~\ref{fig:basis_pursuit_failed_convergence}.
The DR sequence $(y^k)$ for \eqref{problem:basis_pursuit_dual} is pictured transitioning from cyan to magenta over successive iterations.
The DR--induced primal sequence $(x^k)$ is represented using green squares, with the affine subspace $\M$ represented as a green line.
The boundary of squares $[-1,1]^2$ and $[-2,2]^2$ are shown in blue.
We observe DR and its induced method successfully converging in Figure~\ref{fig:basis_pursuit_success}, with success characterised by the sequence $(x^k)$ accumulating at a sparse point on the subspace.
The primal sequence fails to converge to a useful solution to \eqref{problem:basis_pursuit} in Figure~\ref{fig:basis_pursuit_fail_fixed_point} because of DR's apparent attraction to a ``bad fixed point''.
We see in Figure~\ref{fig:basis_pursuit_fail_periodicity} that DR failed to converge to a fixed point and instead settled into a periodic regime.
Repeated application of $T_{D_F,D_G}$ with $\gamma=1$, $U=[1,1]$ and $w=1$ reveals the 6--cycle: $(-1.5,0)$, $(-1.25,-0.25)$, $(-1.25,-0.75)$, $(-1.5,-1)$, $(-1.75,-0.75)$, $(-1.75,-0.25)$. We note that this is a novel instance of DR failing to converge for a $1$--weakly convex problem.
We apparently observe DR iterates \emph{chaotically} orbiting a useful fixed point in Figure~\ref{fig:basis_pursuit_fail_chaos}, resulting in the primal sequence's failure to settle at a solution.
The rightmost image on the second row shows an enlargement of the region we conjecture to be chaotic.
The Cinderella \cite{kortenkamp2000cinderella} script with data for each of these instances is located in the accompanying code.

\begin{figure}[htb!]
    \centering%
    \begin{subfigure}[c]{0.215\textwidth}
\centering
\begin{tikzpicture}[scale=0.36]
\begin{axis}[
  axis lines=middle,
  axis line style = {<->, line width=1.2pt},
  xmin=-2.7, xmax=2.7,
  ymin=-2.7, ymax=2.7,
  xlabel={}, ylabel={},
  xticklabels={}, yticklabels={},
  width=9cm, height=9cm,
  axis equal image,
]

\addplot[blue,line width=1.6pt] coordinates {
  (-2,-2) (2,-2) (2,2) (-2,2) (-2,-2)
};

\addplot[blue, line width=1.6pt] coordinates {
  (-1,-1) (1,-1) (1,1) (-1,1) (-1,-1)
};

\addplot[green!70!black, line width=2.2pt, domain=-2.6:2.6,<->] {-1 + 0.5*x};

\addplot[
  scatter,
  scatter src=explicit,
  scatter/use mapped color={draw=black, fill=mapped color},
  only marks,
  mark=square*,
  mark size=4pt,
  colormap name=teal2green,
]
table[
  col sep=comma,
  x=x,
  y=y,
  meta expr=\coordindex
]{SuccessfulConvergence_ADMM_points.csv};

\addplot[
  scatter,
  scatter src=explicit,
  scatter/use mapped color={draw=black, fill=mapped color},
  only marks,
  mark=*,
  mark size=4pt,
  colormap name=cyan2magenta,
]
table[
  col sep=comma,
  x=x,
  y=y,
  meta expr=\coordindex
]{SuccessfulConvergence_DR_points.csv};

\addplot[
  scatter,
  scatter src=explicit,
  no marks,
  line width = 2pt,
  draw=cyan,
]
table[
  col sep=comma,
  x=x,
  y=y,
  meta expr=\coordindex
]{SuccessfulConvergence_DR_points.csv};

\end{axis}
\end{tikzpicture}
\caption{}
\label{fig:basis_pursuit_success}
\end{subfigure}
\hfill
\begin{subfigure}[c]{0.215\textwidth}
\centering
\begin{tikzpicture}[scale=0.36]
\begin{axis}[
  axis lines=middle,
  axis line style = {<->, line width=1.2pt},
  xmin=-2.7, xmax=2.7,
  ymin=-2.7, ymax=2.7,
  xlabel={}, ylabel={},
  xticklabels={}, yticklabels={},
  width=9cm, height=9cm,
  axis equal image,
]

\addplot[blue, line width=1.6pt] coordinates {
  (-2,-2) (2,-2) (2,2) (-2,2) (-2,-2)
};

\addplot[blue, line width=1.6pt] coordinates {
  (-1,-1) (1,-1) (1,1) (-1,1) (-1,-1)
};

\addplot[green!70!black, line width=2.2pt, domain=-2.6:2.6,<->] {-1 + 0.5*x};

\addplot[
  scatter,
  scatter src=explicit,
  scatter/use mapped color={draw=black, fill=mapped color},
  only marks,
  mark=*,
  mark size=4pt,
  colormap name = cyan2magenta,
]
table[
  col sep=comma,
  x=x,
  y=y,
  meta expr=\coordindex
]{bad_fixed_DR_points.csv};

\addplot[
  scatter,
  scatter src=explicit,
  scatter/use mapped color={draw=black, fill=mapped color},
  only marks,
  mark=square*,
  mark size=4pt,
  colormap name = teal2green,
]
table[
  col sep=comma,
  x=x,
  y=y,
  meta expr=\coordindex
]{bad_fixed_ADMM_points.csv};

\addplot[
  line width = 2pt,
  cyan,
] coordinates {
  (1.98,2.38)
  (2.12,2.05)
};
\end{axis}
\end{tikzpicture}
\caption{}
\label{fig:basis_pursuit_fail_fixed_point}
\end{subfigure}
\hfill
\begin{subfigure}[c]{0.215\textwidth}
\centering
\begin{tikzpicture}[scale=0.36]
\begin{axis}[
  axis lines=middle,
  axis line style = {<->, line width=1.2pt},
  xmin=-2.7, xmax=2.7,
  ymin=-2.7, ymax=2.7,
  xlabel={}, ylabel={},
  xticklabels={}, yticklabels={},
  width=9cm, height=9cm,
  axis equal image,
]

\addplot[blue,line width=1.6pt] coordinates {
  (-2,-2) (2,-2) (2,2) (-2,2) (-2,-2)
};

\addplot[blue, line width=1.6pt] coordinates {
  (-1,-1) (1,-1) (1,1) (-1,1) (-1,-1)
};

\addplot[green!70!black, line width=2.2pt, domain=-1.5:2.6,<->] {1 - x};

\addplot[
  scatter,
  scatter src=explicit,
  scatter/use mapped color={draw=black, fill=mapped color},
  only marks,
  mark=square*,
  mark size=4pt,
  colormap name=teal2green,
]
table[
  col sep=comma,
  x=x,
  y=y,
  meta expr=\coordindex
]{periodic_ADMM_points.csv};

\addplot[
  scatter,
  scatter src=explicit,
  scatter/use mapped color={draw=black, fill=mapped color},
  only marks,
  mark=*,
  mark size=4pt,
  colormap name=cyan2magenta,
]
table[
  col sep=comma,
  x=x,
  y=y,
  meta expr=\coordindex
]{periodic_DR_points.csv};

\addplot[
  scatter,
  scatter src=explicit,
  no marks,
  line width = 2pt,
  draw=cyan,
]
table[
  col sep=comma,
  x=x,
  y=y,
  meta expr=\coordindex
]{periodic_DR_points.csv};

\addplot[
  line width = 2pt,
  cyan,
] coordinates {
  (-1.75,-0.25)
  (-1.5,0)
};
\end{axis}
\end{tikzpicture}
\caption{}
\label{fig:basis_pursuit_fail_periodicity}
\end{subfigure}
\hfill
\begin{subfigure}[c]{0.215\textwidth}
\centering
\begin{tikzpicture}[scale=0.36]
\begin{axis}[
  axis lines=middle,
  axis line style = {<->, line width=1.2pt},
  xmin=-2.7, xmax=2.7,
  ymin=-2.7, ymax=2.7,
  xlabel={}, ylabel={},
  xticklabels={}, yticklabels={},
  width=9cm, height=9cm,
  axis equal image,
]

\addplot[blue, line width=1.6pt] coordinates {
  (-2,-2) (2,-2) (2,2) (-2,2) (-2,-2)
};

\addplot[blue, line width=1.6pt] coordinates {
  (-1,-1) (1,-1) (1,1) (-1,1) (-1,-1)
};

\addplot[green!70!black, line width=2.2pt, domain=-1.25:2.6,<->] {1.14 - 0.951954*x};

\addplot[
  scatter,
  scatter src=explicit,
  scatter/use mapped color={draw=black, fill=mapped color},
  only marks,
  mark=*,
  mark size=4pt,
  colormap name = magenta2cyan,
]
table[
  col sep=comma,
  x=x,
  y=y,
  meta expr=\coordindex
]{chaos_DR_points.csv};

\addplot[
  scatter,
  scatter src=explicit,
  scatter/use mapped color={draw=black, fill=mapped color},
  only marks,
  mark=square*,
  mark size=4pt,
  colormap name = teal2green,
]
table[
  col sep=comma,
  x=x,
  y=y,
  meta expr=\coordindex
]{chaos_ADMM_points.csv};

\addplot[
  scatter,
  scatter src=explicit,
  no marks,
  line width=2pt,
  draw=cyan,
]
table[
  col sep=comma,
  x=x,
  y=y,
  meta expr=\coordindex
]{chaos_DR_points.csv};
\end{axis}
\end{tikzpicture}
\caption{}
\label{fig:basis_pursuit_fail_chaos}
\end{subfigure}
\hfill
\begin{subfigure}[c]{0.1\textwidth}
\centering
\begin{tikzpicture}[scale=0.38]
\begin{axis}[
  axis lines=none,
  xmin=-2.5, xmax=-0.5,
  ymin=-1.5, ymax=0.3,
  xticklabels={}, yticklabels={},
  width=9cm, height=9cm, 
  axis equal image,
  clip=true,
]

\addplot[
  scatter,
  scatter src=explicit,
  scatter/use mapped color={draw=black, fill=mapped color},
  only marks,
  mark=*,
  mark size=4pt,
  colormap name = magenta2cyan,
]
table[col sep=comma, x=x, y=y, meta expr=\coordindex]
{chaos_DR_points.csv};

\addplot[
  no marks,
  line width=2pt,
  draw=cyan,
]
table[col sep=comma, x=x, y=y]
{chaos_DR_points.csv};

\end{axis}
\end{tikzpicture}
\caption*{}
\end{subfigure}
    \caption{DR and induced primal sequences generated for instances of \eqref{problem:basis_pursuit_dual}.}%
    \label{fig:basis_pursuit_failed_convergence}%
\end{figure}
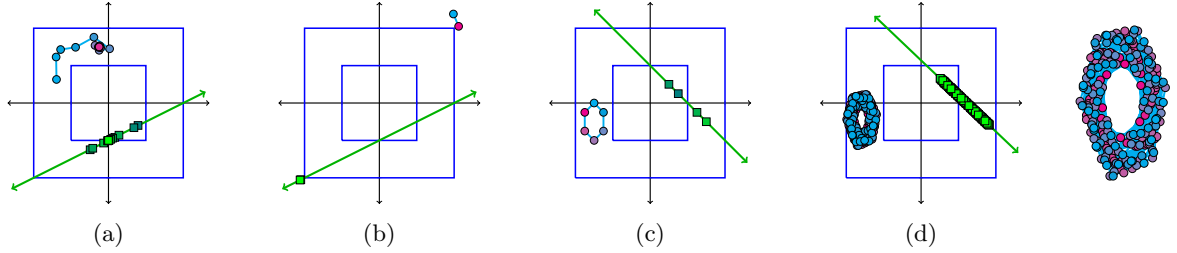

\section{A primal--dual scheme derived from Gabay duality}\label{sec:prim_dual_scheme}

To address DR's occasional failure to converge for non maximal monotone problems, we propose a novel splitting scheme based on Gabay duality. We analyse the method for a class of operator slightly more general than maximal monotonicity:
\begin{definition}[Operators with nonexpansive selectants]
    Let $D:\H\rightrightarrows\H$.
    We say that $D\in\NR$ if there exists some $\gamma>0$ such that the selectant $S_{\gamma D}$ is nonexpansive. 
\end{definition}
Class $\NR$ contains maximal monotone operators.
Indeed, the resolvents of maximal monotone operators are nonexpansive, so one may simply take $S_{\gamma D}=J_{\gamma D}$.
There also exist some non maximal monotone operators that possess nonexpansive selectants, such as the inverse Clarke subdifferentials of the MCP and SCAD.
The fact that MCP and SCAD belong to this class (for certain parameter values) is shown in Appendix~\ref{app:mcp_scad_derivations}.

If an operator's resolvent is nonexpansive but not firmly, its reflection will fail to be nonexpansive.
We suspect DR's observed difficulties with \eqref{problem:basis_pursuit_dual} are a result of $R_{\gamma (\partial \reg)^{-1}}$ not being nonexpansive.
The method we propose is instead based around \emph{over--relaxed resolvents}.
\begin{definition}[Over--relaxed resolvent]\label{def:over_relaxed_resolvent}
    Let $D\colon\H\rightrightarrows\H$ and $\alpha \in [0,1]$.
    The over--relaxed resolvent of $D$ is defined as $R_{\gamma D}^{\alpha} \coloneqq (1+\alpha) J_{\gamma D} - \alpha I$.
\end{definition}
Over--relaxed resolvents are related to reflections in the following way.
\begin{lemma}\label{lem:over_relaxations_H_theta}
    Let $D\colon\H\rightrightarrows\H$.
    For any $\alpha\in[0,1]$, there exists a possibly set-valued mapping $D^{\alpha}$ and associated resolvent $J_{\gamma D^{\alpha}}$ such that $R^{\alpha}_{\gamma D} = R_{\gamma D^{\alpha}}$. 
\end{lemma}
\begin{proof}
    We set $D^{\alpha}=\gamma^{-1}(\frac12(1+\alpha)J_{\gamma D}+\frac12 (1-\alpha)I)^{-1}-\gamma^{-1}I$.
    It follows that $I + \gamma D^{\alpha} = (\frac12(1+\alpha)J_{\gamma D}+\frac12 (1-\alpha)I)^{-1}$, wherefore the inverse mappings are equal:
    $\frac12(1+\alpha)J_{\gamma D}+\frac12 (1-\alpha)I=(I+\gamma D^{\alpha})^{-1}=J_{\gamma D^{\alpha}}$.
    Rearranging yields $(1+\alpha)J_{\gamma D} - \alpha I=2J_{\gamma D^{\alpha}}-I$, as required.
\end{proof}

\begin{remark}\label{rem:homotopy_operator_exists}
    If we consider the dual operator $D_H=(-M\circ H^{-1}\circ -M^*)-d$ with $M=-I$, Lemma~\ref{lem:over_relaxations_H_theta} grants the existence of $H^{\alpha}$ such that $R^{\alpha}_{\gamma (H^{-1}-d)}=R_{\gamma ((H^{\alpha})^{-1}- d)}$.
    The operator $D^{\alpha}$ may enjoy a natural interpretation as a \emph{homotopy} related to $D$.
    For instance, if $D=(\partial \phi_{\mcp})^{-1}$ then $D^{\alpha}=(\partial\phi^{\alpha}_{\mcp})^{-1}$. Similarly, $D^{\alpha}=(\partial\phi_{\scad}^{\alpha})^{-1}$ corresponds to $D=(\partial \phi_{\scad})^{-1}$. The mappings $\phi_{\mcp}^{\alpha}$ and $\phi_{\scad}^{\alpha}$ are visualised in Figure~\ref{fig:homotopies}.
    \begin{figure}[tbh]
        \centering%
        \begin{subfigure}{0.14\textwidth}
    \centering
    \begin{tikzpicture}[domain=-3:3]
        \node[anchor=south,inner sep=0] (image) at (0,0){\includegraphics[width=0.4\columnwidth]{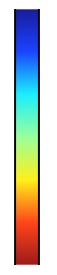}};
        \node[font=\small] at (-0.5,4.2){$1$};
        \node[font=\small] at (-0.5,2.2){$\alpha$};
        \node[font=\small] at (-0.5,0.2){$0$};
    \end{tikzpicture}
\end{subfigure}
\hfill
\begin{subfigure}{0.4\textwidth}
    \centering
    \begin{tikzpicture}[domain=-3:3]
        \draw[line width = 0.6pt,->] (0,-0.1) -- (0,4.8) node[pos = 1,right] {};
        \draw[line width = 0.6pt,->] (-2.3,0) -- (2.3,0) node[pos=0,above right] {} node[pos=1,above left] {};
        \node[anchor=south] (image) at (0,0){\includegraphics[width=0.7\columnwidth]{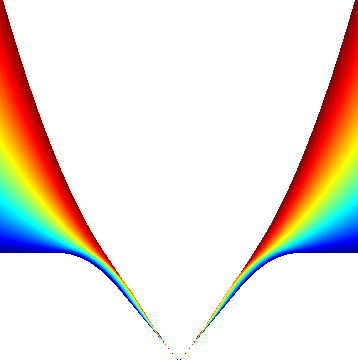}};
        \node[anchor=south west] (capt) at (0.2,2.2) {$\phi_{\scad}^{\alpha}$};
    \end{tikzpicture}
\end{subfigure}
\hfill
\begin{subfigure}{0.4\textwidth}
    \centering
    \begin{tikzpicture}[domain=-3:3]
        \draw[line width = 0.6pt,->] (0,-0.1) -- (0,4.8) node[pos = 1, right] {};
        \draw[line width = 0.6pt,->] (-2.3,0) -- (2.3,0) node[pos=0,above right] {} node[pos=1,above left] {};
        \node[anchor=south] (image) at (0,0){\includegraphics[width=0.7\columnwidth]{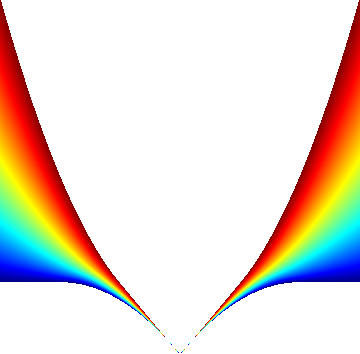}};
        \node[anchor=south west] (capt) at (0.2,2.2) {$\phi_{\mcp}^{\alpha}$};
    \end{tikzpicture}
\end{subfigure}
\caption{The graphs of homotopies $\phi_{\scad}^{\alpha}$ and $\phi_{\mcp}^{\alpha}$ induced by over--relaxed resolvents.}%
\label{fig:homotopies}%
\end{figure}
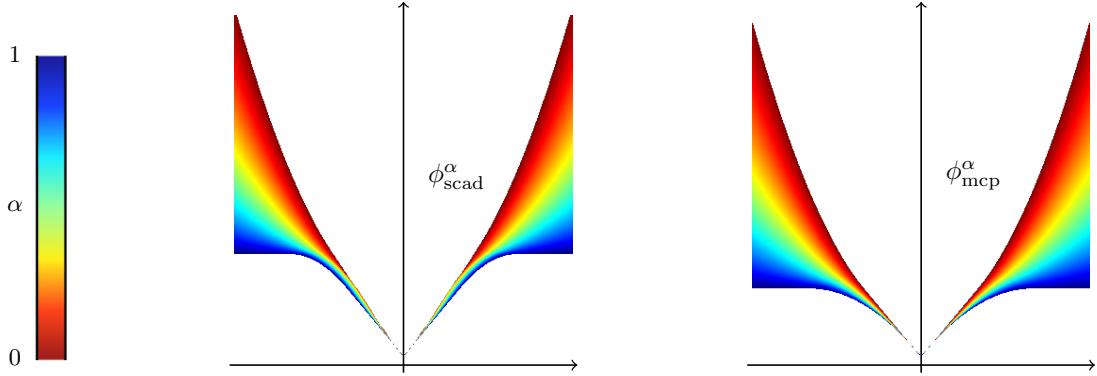
\end{remark}

\begin{corollary}
    Let $D\in\NR$.
    Then 
    \begin{enumerate}[label=(\roman*)]
        \item\label{cor:continuous_in_param_reflections} the mapping $[0,1]\to \H$, $\alpha\mapsto R^{\alpha}_D(y)$ is continuous for any arbitrary $y\in\H$;
        \item\label{lem:reflections_Lipschitz} there exists $D^{\alpha}$ such that $R_{\gamma D^{\alpha}}$ is $(1+2\alpha)$--Lipschitz.
    \end{enumerate}
\end{corollary}
\begin{proof}
    Item \ref{cor:continuous_in_param_reflections} follows immediately from fact that over--relaxed resolvents are affine in $\alpha$.
    Considering \ref{lem:reflections_Lipschitz}, the existence of $D^{\alpha}$ such that $R_{\gamma D^{\alpha}}=R^{\alpha}_{\gamma D}$ is given by Lemma~\ref{lem:over_relaxations_H_theta}.
    Then, for all $x,y\in\H$,
    \begin{multline*}
        \left\| R_{\gamma D^{\alpha}}(x) - R_{\gamma D^{\alpha}}(y) \right\|
        =
        \left\| R_{\gamma D}^{\alpha}x - R_{\gamma D}^{\alpha}y \right\|
        =
        \left\| (1+\alpha)(J_{\gamma D}x-J_{\gamma D}y)+\alpha(y-x) \right\| \\
        \leq
        (1+\alpha) \left\| J_{\gamma D}x-J_{\gamma D}y \right\| + \alpha\|x-y\|
        \leq
        (1+2\alpha) \|x-y\|,
    \end{multline*}
    as claimed, where the last inequality follows from $D\in\NR$.
\end{proof}

\begin{proposition}\label{prop:descends_to_quotient}
    Let $M:\H \rightarrow \H_1$ and $J:\H_1 \rightarrow \H_1$ be a set valued mapping. Let $D_H=(-M\circ H^{-1}\circ -M^*)(\cdot)- d$. The following are equivalent.
    \begin{enumerate}[label=(\roman*)]
        \item\label{item:H_eps_existence} There exists a set-valued mapping $H:X \rightarrow X$ such that $(H)^{-1} :{\mathrm {Image}}M^* \rightarrow \H$ satisfies $J = (I+\gamma D_{H})^{-1}$;        \item\label{item:descends_to_quotient} $(J)^{-1}=\tilde{J}+I$ for some $\tilde{J}: \H_1 \rightrightarrows {\mathrm{Image}}M$ such that $\tilde{J}$ descends to the quotient $\H_1/\ker M^*$ (i.e. whenever $u_1 - u_2 \in \ker M^*$, $\tilde{J}(u_1) = \tilde{J}(u_2)$).
    \end{enumerate}
\end{proposition}
\begin{proof}
    (\ref{item:H_eps_existence} $\implies$ \ref{item:descends_to_quotient}).
    Suppose \ref{item:H_eps_existence} holds.
    Then $J$ and $(I+\gamma D_{H})^{-1}: \H_1 \rightrightarrows \H_1$ are equal.
    Hence their `inverse' maps $\left(J\right)^{-1}$ and $I+\gamma D_{H}: \H_1 \rightrightarrows \H_1$ must also be equal, i.e., $\left(J\right)^{-1}=I+\gamma D_H$.
    Thus we obtain  
    \begin{equation}\label{eqn:reformulate_H_eps_problem}
    \gamma (-M\circ (H)^{-1}\circ -M^*)(\cdot)-\gamma d= \left( J \right)^{-1} - I.
    \end{equation}
    It follows that $(J)^{-1}= \tilde{J} + I$ for some $\tilde{J}: \H_1 \rightrightarrows {\mathrm{Image}}M$.
    Let $-M^*u_1 = -M^*u_2$.
    As $(H)^{-1}$ is well-defined by assumption, we have that 
    \begin{alignat*}{2}
    (-M^*u_1 = -M^*u_2) \implies&& U_1\coloneqq (H)^{-1}(-M^*u_1) = (H)^{-1}(-M^*u_2) =:U_2,
    \end{alignat*}
    where \eqref{eqn:reformulate_H_eps_problem} makes clear that
    \begin{align*}
        (H)^{-1}(-M^*u_1) &= U_1 \quad \text{satisfies}\quad -\gamma M(U_1)-\gamma d = (J)^{-1}(u_1)-u_1\\
        \text{and}\quad (H)^{-1}(-M^*u_2) &= U_2 \quad \text{satisfies}\quad -\gamma M(U_2) -\gamma d= (J)^{-1}(u_2)-u_2.
    \end{align*}
    Since $U_1=U_2$, we must have $-\gamma M(U_1)-\gamma d=-\gamma M(U_2)-\gamma d$.
    It follows that
    \begin{multline*}
        (J)^{-1}(u_1)-u_1 = (J)^{-1}(u_2)-u_2 \\
    \implies \tilde{J}(u_1) + u_1 - u_1 = \tilde{J}(u_2) + u_2-u_2
    \implies \tilde{J}(u_1) = \tilde{J}(u_2),
    \end{multline*}
    which shows \ref{item:descends_to_quotient}.
    
    (\ref{item:descends_to_quotient} $\implies$ \ref{item:H_eps_existence}).
    Suppose \ref{item:descends_to_quotient}.
    We first show that there exists a well-defined mapping $(H)^{-1}$ such that
    $$
    (H)^{-1} \circ-M^* = (-M)^{-1}( \gamma^{-1}(\left( J \right)^{-1}-I))+d.
    $$
    Let $-M^*u_1 = -M^*u_2$.
    Since $\tilde{J}$ descends to the quotient, it holds that $\tilde{J}(u_1) = \tilde{J}(u_2)$, whereupon it is clear that $(H)^{-1} (-M^*u_1) = (H)^{-1} (-M^*u_2)$.
    This shows $(H)^{-1}$ is well-defined.
    Now for any $u$ we have that
    \begin{multline*}
    (H)^{-1} (-M^* u) = (-M)^{-1} (\gamma^{-1}((J)^{-1}u-u)+d\\
    \implies -M((H)^{-1} (-M^*u)) = \gamma^{-1}(J)^{-1}(u) - \gamma^{-1} u +d\\
    \implies ((-\gamma M\circ(H)^{-1}\circ -M^*) + I)(u)-\gamma d = (J)^{-1}(u).
    \end{multline*}
    As the two set-valued mappings are equal, so also must be their associated inverse mappings.
    This shows \ref{item:H_eps_existence}.
\end{proof}
We note that condition \ref{item:descends_to_quotient} holds whenever $M^*$ is injective.

\subsection{Homotopically--stabilised DR}\label{sec:HOST}

We propose the \emph{homotopically stabilised Douglas--Rachford method} (HOST).
The method uses the following operator. 
\begin{definition}[HOST operator]\label{def:HOST_operator}
    Let $(\varphi^k)$ and $(\theta^k)$ be sequences in $[0,1]$ and $\A,\B:\H\rightrightarrows\H$ have computable resolvents.
    The HOST operator is defined as
    $$
    T^{\varphi^k,\theta^k}_{\A,\B}\coloneqq \frac12\left(I+R^{\varphi^k}_{\gamma\A}\circ R^{\theta^k}_{\gamma\B}\right).
    $$
\end{definition}
It follows from Lemma~\ref{lem:over_relaxations_H_theta} that recursively applying the HOST operator is equivalent to DR with operator $T_{\A^{\varphi^k},\B^{\theta^k}}$.
If $\A$ and $\B$ are maximal monotone, $(\varphi^k,\theta^k)\to (1,1)$ and $\Fix T_{\A,\B}\neq\emptyset$, we have by Theorem~\ref{thm:DR} that any sequence defined by $y^{k+1}\in T_{\A,\B}^{\varphi^k,\theta^k}(y^k)$ converges to $u\in\Fix T^{1,1}_{\A,\B}$ such that $J_{\gamma \B}u\in\zeroset(\A+\B)$.
We use the HOST operator to propose a convergent scheme for \eqref{general_inclusion_problem} under the following assumption.
\begin{assumption}
    Operators $\A,\B\in\NR$ and there exists some $\lambda^*\in\zeroset(\A+\B)$.
\end{assumption}
Theorem~\ref{thm:GMI_duality} makes possible many algorithmic schemes for \eqref{primal_inclusion_problem} built upon DR--type methods.
We have already mentioned two such strategies that are typically deployed with the aim of reducing the number of computed iterates \cite{lindstrom2022computable,poon2019trajectory}.
In this section, we will use Gabay's duality to motivate the design of an algorithm with improved stability for nonmonotone problems.

Before we introduce the prototypical scheme, some remarks are in order.
The method is designed to always converge with one of two outcomes:
\begin{enumerate}[label=\Roman{*}.,ref=\Roman{*}]
    \item The solution obtained may only be related to the original problem.
    In this case, the practitioner receives an accompanying certificate ($\tau = 0$) together with some information about the problem that was actually solved.
    Said information comes in the form of limits for the parameter sequences used $(\varphi^k,\theta^k)$---for many problems, these values admit clear interpretations and the solutions obtained are still of practical value for the problem we set out to solve (c.f.~the homotopic regularizers in section~\ref{subsec:rlad}). 
    \item\label{item:the_good_situation} The solution obtained solves \eqref{primal_inclusion_problem}, in which case $\tau=1$ and $(\varphi^k,\theta^k) \to (1,1)$.
    This is the case that requires the most demanding analysis.
\end{enumerate}
In practice, we observed that for each of our experiments situation \ref{item:the_good_situation} occurred.
This is as we desired, and also unsurprising when one understands the motivations for our algorithmic scheme.
In essence, we do as follows.
First, we choose our sequences in such a way that we begin by applying $T^{0,0}_{\A,\B} = \frac12 I+\frac12 J_{\gamma \A}\circ J_{\gamma \B}$.
    This mapping produces a globally convergent sequence (cf. Theorem~\ref{thm:DR}) under the assumption $\A,\B\in\NR$.
    During this ``warm start'' period, we hope to draw the sequence $(y^k)$ away from any regions of pathological behaviour. 
Second, if our iterates $y^k$ appear to be converging with a rate no worse than a sublinear tolerance we specify, then we update our parameters so that the HOST operator $T^{\varphi^k,\theta^k}_{\A,\B}$ converges pointwise to the DR operator $T_{\A,\B}$.
By asking for a rate that is no worse than sublinear, we aim to give the iterates $y^k$ ample opportunity to follow the moving solution to the evolving problem.
Since DR is well--known for exhibiting linear convergence, it is reasonable to expect that this strategy would work well, as it did in our experiments.

While we are able to prove HOST's convergence to a solution of \eqref{primal_inclusion_problem}, in the \emph{optimisation setting} we do not know sufficient conditions to guarantee that the global minimizer to the evolving problem moves continuously along a path.
Such guarantees, even for one specific chosen regularizer, would necessitate a separate investigation.
Consider, for example, \cite{hoheisel2023lasso}, wherein the authors provided sufficient conditions for the global solution to classical LASSO ($\min_x \; \|Ax-b\|^2 + \gamma \| x\|_1$) to be continuous in $\gamma$.
This analysis was sufficiently involved as to constitute an independent study.
Notwithstanding, the high quality solutions we obtained are evidence that the global minimizer may move continuously for some problems.

The proposed HOST scheme is detailed in Algorithm~\ref{alg:homotopy}.
As is typical, the algorithm's safety termination condition depends on a user-selected maximum number of iterations.
The method may otherwise terminate if several criteria are met.
The first is based on the closeness of the homotopy parameters to 1, indicating that the evolving problem is sufficiently close to the target problem. 
This is combined with the standard successive iterate difference criterion $\|y^{k+1}-y^k\|\leq\tol_y$.

\begin{algorithm}[htb!]
    \caption{HOST: Homotopically stabilized Douglas--Rachford splitting}
    \label{alg:homotopy}
    \begin{algorithmic}[1]
        \State Choose $\beta\geq0$ and $\bar{p} >0$ and build
        sequence $(p^k)$ according to $p^k=\beta/k^{\bar{p}+1}$
        \State Fix resolvent parameter $\gamma>0$
        \State Specify
        sequences $(\hat{\varphi}^k),(\hat{\theta}^k) \subset \left[0,1\right]$ such that $\hat{\varphi}^k,\hat{\theta}^k \uparrow 1$ and $R^{\hat{\varphi}^0}_{\gamma \A}$, $R^{\hat{\theta}^0}_{\gamma \B}$ are nonexpansive
        \State Specify a starting point $y^0$ and tolerance $\tol_y>0$ 
        \State Specify reflection parameter tolerances $\tol_{\varphi},\tol_{\theta}\in [0,1]$
        \State Set $(\varphi^0,\theta^0) \gets (\hat{\varphi}^0,\hat{\theta}^0)$ and initialize $(\tau,j) \gets (1,0)$
        \For{$k=0,1,\ldots,k_{\max}$}
            \State $T^k \gets \frac12 \left( I + R_{\gamma \A}^{\varphi^k}R_{\gamma \B}^{\theta^k} \right)$, $y^{k+1} \gets T^k (y^k)$
            \If{$\|y^{k+1} - y^k\| \leq p^k$} \Comment{(If the Cauchy condition is satisfied at step $k+1$)}
                \If{$\tau = 1$} \Comment{(And if the Cauchy condition has never been violated)}
                    \State $j \gets j + 1$
                    \State $(\varphi^{k+1},\theta^{k+1}) \gets (\hat{\varphi}^j,\hat{\theta}^j)$ \Comment{(Increment sequences closer to 1)}
                \Else \Comment{(If the Cauchy condition has ever been violated)}
                    \State $(\varphi^{k+1},\theta^{k+1}) \gets ({\varphi}^{k},{\theta}^{k})$ \Comment{(Keep the sequences the same)}
                \EndIf
            \Else \Comment{(If the Cauchy condition is violated at step $k+1$)}
                \State $\tau \gets 0$, $j \gets \max\{j - 1,0\}$
                \State $(\varphi^{k+1},\theta^{k+1}) \gets (\hat{\varphi}^j,\hat{\theta}^j)$ \Comment{(Backtrack toward the nonexpansive case)}
            \EndIf
            \State $\tau^k \gets \tau$
            \If{$\varphi^k\geq1-\tol_{\varphi},\;\theta^k\geq1-\tol_{\theta}$ and $\|y^{k+1}-y^k\|\leq \tol_y$}
            \State \Return $y^{k+1}$
            \EndIf
        \EndFor
        \State \Return $\lambda\in J_{\gamma\B}(y^{k+1})$
    \end{algorithmic}
\end{algorithm}

\subsubsection{Convergence analysis}\label{subsec:convergence_analysis}

We show in the following Lemma the convergence of the HOST sequence $(y^k)$.
There is no analytic subtlety here, as we choose the governing sequence $(p^k)$ to enforce that the dual sequence $(y^k)$ either be Cauchy or be generated by an averaged mapping.
The actual subtlety lies in Theorems~\ref{thm:dual_convergence} and \ref{thm:primal_convergence}, where we unpack the properties of the limit obtained.

\begin{lemma}\label{lem:homotopy_alg_convergence}
    The following hold:
    \begin{enumerate}[label=(\roman*)]
        \item\label{item:parameters_converge} The sequence $(\varphi^k,\theta^k)$ in Algorithm~\ref{alg:homotopy} converges to some $(\bar{\varphi},\bar{\theta})\in [0,1]^2$.
        \item\label{item:sequence_converges} The sequence $(y^k)$ converges weakly. If in addition $(\varphi^k,\theta^k)$ converges to any value other than $(\varphi^0,\theta^0)$, then $(y^k)$ converges strongly.
        \item\label{item:operators_converge} The operator sequence $(T^{\varphi^k,\theta^k}_{\A,\B})$ is pointwise convergent to
    $$
    \overline{T}\coloneqq \frac12 \left( I+R_{\gamma \A}^{\bar{\varphi}}\circ R_{\gamma \B}^{\bar{\theta}}\right) .
    $$
    \end{enumerate}
\end{lemma}
\begin{proof}
    \ref{item:parameters_converge}. By construction, one of the following holds:
    \begin{enumerate}
        \item\label{item:tau_goes_to_1} Flag $\tau^k = 1$ for all $k$. In this case, we have that $(\varphi^k,\theta^k) \to (\bar{\varphi},\bar{\theta}) = (1,1)$ by construction of $(\hat{\varphi}^k,\hat{\theta}^k)$.
        Moreover, the condition $\| y^{k+1}-y^k\| \leq p^k$ is never violated, whereupon $\sum_k \| y^{k+1}-y^k\| < \infty$ and $(y^k)$ is Cauchy.
        \item\label{item:tau_goes_to_0} There exists $K$ such that $\tau^k =0$ for all $k\geq K$. For $k \geq K$ we have that the sequences $(\varphi^k,\theta^k)$ are monotone decreasing and that $\{(\varphi^k,\theta^k)\}_{k \geq K} \subseteq \{ (\hat{\varphi}^j,\hat{\theta}^j) \}_{j \leq K}$, which is a finite set.
        By the pigeonhole principle, $(\varphi^k,\theta^k)$ converges finitely (after, say, $K^\prime$ iterations) to some $(\bar{\varphi},\bar{\theta}) \in \{ (\hat{\varphi}^j,\hat{\theta}^j) \}_{j \leq K}$.
    \end{enumerate}
    
    \ref{item:sequence_converges}.
    If $(\bar{\varphi},\bar{\theta}) \neq ({\varphi}^0,{\theta}^0)$ then we have that the condition $\|y^{k+1}-y^k \| \leq p^k$ is violated at most finitely many times (specifically, it is never violated for $k\geq K^\prime$), wherefore $\sum_k \|y^{k+1}-y^k \| < \infty$, wherefore $y^k$ is Cauchy and therefore converges.
    If $(\bar{\varphi},\bar{\theta}) = ({\varphi}^0,{\theta}^0)$, then, say after $K^\prime$ iterations, $T^{\varphi^k,\theta^k}_{\A,\B}=T_{\A,\B}^{\varphi^0,\theta^0}$, which by Theorem~\ref{thm:DR} admits a sequence $y^k$ that converges weakly to a fixed point of $T^{\varphi^0,\theta^0}_{\A,\B}$.

    \ref{item:operators_converge}.
    The convergence of $(\varphi^k,\theta^k)$ to some $(\bar{\varphi},\bar{\theta})$ follows from \ref{item:parameters_converge}.
    The statement then follows from Corollary~\ref{cor:continuous_in_param_reflections}.
\end{proof}

\begin{theorem}[Convergence of HOST]\label{thm:dual_convergence} 
    Suppose one of the following:
    \begin{enumerate}
        \item $\H$ is finite dimensional;
        \item $(\bar{\varphi},\bar{\theta}) \neq ({\varphi}^0,{\theta}^0)$.
    \end{enumerate}
    Then the sequence $(y^k)$ generated by Algorithm~\ref{alg:homotopy} converges to a fixed point of $\overline{T}$.
    If $(\bar{\varphi},\bar{\theta}) \neq ({\varphi}^0,{\theta}^0)$, then the rate of this convergence is no worse than $\frac{\beta}{\bar{p}k^{\bar{p}}}$.
\end{theorem}
\begin{proof}
    If $(\bar{\varphi},\bar{\theta}) = ({\varphi}^0,{\theta}^0)$, then the convergence holds due to Theorem~\ref{thm:DR}.
    Let's now suppose otherwise.
    Since $(\bar{\varphi},\bar{\theta}) \neq (\varphi^0,\theta^0)$, we have by Lemma~\ref{lem:homotopy_alg_convergence} that the limit $\bar{y}=\lim_{k\to\infty} y^k$ exists and that, for some $K^\prime\in\N$,
    $\|y^k-y^{k+1}\|\leq \beta/k^{\bar{p}+1}$ for all $k\geq K^\prime$.
    It follows that for all $k> K^\prime$:
    \begin{equation}\label{eqn:sublinear_yield}
        \|\bar{y}-y^k\|\leq \sum_{j=k+1}^{\infty}\|y^j-y^{j-1}\|\leq\sum_{j=k+1}^{\infty} \frac{\beta}{j^{\bar{p}+1}}\leq \int_k^{\infty} \frac{\beta}{x^{\bar{p}+1}}dx=\frac{\beta}{\bar{p}k^{\bar{p}}}.
    \end{equation}
    It remains to show that $\bar{y}$ is a fixed point of $\overline{T}$. 
    The triangle inequality yields
    $$
        \left\| \overline{T}\bar{y}-\bar{y} \right\|
        \leq
        \left\| y^k-\bar{y} \right\|+ \left\| \overline{T}\bar{y}-T^{\varphi^k,\theta^k}_{\A,\B}\bar{y} \right\| + \left\| y^k-T^{\varphi^k,\theta^k}_{\A,\B}y^k \right\| + \left\| T^{\varphi^k,\theta^k}_{\A,\B}\bar{y}-T^{\varphi^k,\theta^k}_{\A,\B}y^k \right\|.
    $$
    We aim to show that each term of the right hand side above approaches zero as $k\rightarrow \infty$.
    The first is clear.
    The second term goes to zero by the pointwise convergence of $T^{\varphi^k,\theta^k}_{\A,\B}$ to $\overline{T}$ (Lemma~\ref{lem:homotopy_alg_convergence}\,\ref{item:operators_converge}).
    The convergence of the third term follows from the fact that $(y^k)_{k\in \N}$ is Cauchy and $y^{k+1}=T^{\varphi^k,\theta^k}_{\A,\B}(y^k)$.
    To show that the final term tends to zero, we first derive a Lipschitz constant for $T^{\varphi^k,\theta^k}_{\A,\B}$.
    Letting $x,y\in \H$ we have
    \begin{align*}
        \big\|T^{\varphi^k,\theta^k}_{\A,\B}&x-T^{\varphi^k,\theta^k}_{\A,\B}y\big\|
    =
        \frac12\big\|R_{\gamma \A}^{\varphi^k}R_{\gamma \B}^{\theta^k}x-R_{\gamma \A}^{\varphi^k}R_{\gamma \B}^{\theta^k}y+x-y\big\| \\
    &\leq
        \frac12\big\|(1+\varphi^k)\left( J_{\gamma\A}R^{\theta^k}_{\gamma\B}x-J_{\gamma\A}R^{\theta^k}_{\gamma\B}y\right) -\varphi^k\left( R^{\theta^k}_{\gamma\B}x-R^{\theta^k}_{\gamma\B}y\right) \big\|+\frac12\|x-y\| \\
    &\leq
        \frac{1+\varphi^k}{2}\big\|J_{\gamma\A}R^{\theta^k}_{\gamma\B}x-J_{\gamma\A}R^{\theta^k}_{\gamma\B}y\big\|+\frac{\varphi^k}{2}\big\|R^{\theta^k}_{\gamma\B}x-R^{\theta^k}_{\gamma\B}y\big\|+\frac12\|x-y\| \\
    &\leq
        \frac{1+\varphi^k}{2}\big\|R^{\theta^k}_{\gamma\B}x-R^{\theta^k}_{\gamma\B}y\big\|+\frac{\varphi^k}{2}\big\|R^{\theta^k}_{\gamma\B}x-R^{\theta^k}_{\gamma\B}y\big\|+\frac12\|x-y\| \\
    &=
        \left( \frac12+\varphi^k\right) \big\|(1+\theta^k)\left( J_{\gamma\B}x-J_{\gamma\B}y\right) -\theta^k(x-y)\big\|+\frac12\|x-y\| \\
    &\leq
        \left( \frac{1+\theta^k}{2}+\varphi^k+\varphi^k\theta^k \right) \big\|J_{\gamma\B}x-J_{\gamma\B}y\big\|+\left( \frac{\theta^k+1}{2}+\varphi^k\theta^k\right) \|x-y\| \\
    &\leq
        (1+\theta^k+\varphi^k+2\varphi^k\theta^k)\|x-y\|.
    \end{align*}
    Combining this estimate with \eqref{eqn:sublinear_yield}, we see that for $k>K^\prime$ we have
    $$
        \big\|T^{\varphi^k,\theta^k}_{\A,\B}\bar{y}-T^{\varphi^k,\theta^k}_{\A,\B}y^k\big\|
        \leq
        \frac{\beta}{\bar{p}k^{\bar{p}}} (1+\varphi^k+\theta^k+2\varphi^k\theta^k),
    $$
    with the right hand side converging to zero as $k\rightarrow\infty$,
    thus showing that $\bar{y}$ is a fixed point of $\overline{T}$.
\end{proof}

We now focus our attention to the specific case that $\A=D_F$ and $\B=D_G$, with the intention of connecting HOST to a primal algorithm generalizing ADMM.
We first establish the following Lemma.

\begin{lemma}[Characterization of induced resolvents]\label{lem:characterisation_of_induced_resolvent}
    Let $D_H\in\NR$ be such that there exists $(H^{\alpha})^{-1}:\H\rightarrow\H$ such that $R^{\alpha}_{\gamma D_H} = R_{\gamma D_{H^{\alpha}}}$.\footnote{If $M=\pm I$ or $M$ is injective, then this assumption holds (see Remark~\ref{rem:homotopy_operator_exists} and Proposition~\ref{prop:descends_to_quotient}).}
    Let $M:\H\rightarrow \H_1$ be injective, linear, and $M^{-1}$ denote its preimage map.
    Then the following hold.
    \begin{enumerate}[label=(\roman*)]
        \item\label{item:nonexpansive_homotopy_resolvent} The induced resolvent $J_{\gamma D_{H^{\alpha}}}$ is nonexpansive.
        
        \item\label{item:continuous_homotopy_resolvent} For any $y\in\H$, the map $\alpha\mapsto J_{\gamma D_{H^{\alpha}}}(y)$ is single--valued and continuous. 
    
        \item\label{item:primal_rearrangement} $  (\gamma M\circ J^M_{\gamma^{-1}H^{\alpha}}\circ-M^*)(\gamma^{-1}(\cdot) + d_H)=\frac12(1+\alpha)(J_{\gamma D_H} - I) -\gamma d_H$.
        \item \label{item:primal_continuous_in_parameters} For any fixed $u \in \H_1$, the map
            \begin{equation}
            \left[0,1\right] \rightarrow {\mathrm{Image}}(M),\quad \alpha \mapsto J^M_{\gamma^{-1}H^{\alpha}}( -M^*(\gamma^{-1}u + d_H))
            \label{primal_resolve_in_alpha}
            \end{equation}
            is single--valued and continuous.
            
            \item\label{item:J_is_Lipschitz} Operator $( J^M_{\gamma^{-1} H^{\alpha}}\circ -M^*) (\gamma^{-1}(\cdot)+d_H)$ is Lipschitz continuous with modulus $(1+\alpha)/\gamma\sigma_{\min}(M)$. Here $\sigma_{\min}(M)$ denotes the smallest singular value of $M$.
    \end{enumerate}
\end{lemma}
\begin{proof}
    \ref{item:nonexpansive_homotopy_resolvent}.
    By assumption we have $J_{\gamma D_{H^{\alpha}}}(y)=\frac12((1+\alpha)J_{\gamma D_H}(y)+(1-\alpha)y)$.
    It follows that for all $x,y\in\H$:
    \begin{align*}
        \big\|J_{\gamma D_{H^{\alpha}}}x-J_{\gamma D_{H^{\alpha}}}y\big\|
        &=
        \frac12\|(1+\alpha)(J_{\gamma D_H}x-J_{\gamma D_H}y)+(1-\alpha)(x-y)\|\\
        &\leq
        \frac{1+\alpha}2\|J_{\gamma D_H}x-J_{\gamma D_H}y\|+\frac{1-\alpha}2\|x-y\|
        \leq
        \|x-y\|,
    \end{align*}
    with the final inequality following from the nonexpansiveness of $J_{\gamma D_H}$.
    
    \ref{item:continuous_homotopy_resolvent}.
    The mapping's single--valuedness follows from the single--valuedness of $J_{\gamma D_H}$, with $J_{\gamma D_{H^{\alpha}}}$ affine in $\alpha$ and hence continuous.
    
    \ref{item:primal_rearrangement}.
    Assuming the existence of $(H^{\alpha})^{-1}$ such that $R^{\alpha}_{\gamma D_H} = R_{\gamma D_{H^{\alpha}}}$, we have
    \begin{equation*}
            \frac{1}{2}\left( (1+\alpha)J_{\gamma D_H} + (1-\alpha)I\right)
            =
            J_{\gamma D_{H^{\alpha}}}
            =
            I + \gamma d_H + \gamma MJ_{\gamma^{-1} H^{\alpha}}^M (-M^*(\gamma^{-1}(\cdot) + d_H)),
    \end{equation*}
    with the second equality following from Theorem~\ref{thm:Moreaus_identity}.
    Rearranging yields \ref{item:primal_rearrangement}.
    
    \ref{item:primal_continuous_in_parameters}.
    The single--valuedness of the right hand side of \ref{item:primal_rearrangement} together with injectivity of $M$ yields the single--valuedness of \eqref{primal_resolve_in_alpha}. The mapping's continuity in $\alpha$ is obvious.
    
    \ref{item:J_is_Lipschitz}.
    By \ref{item:primal_rearrangement} we have
    \begin{align*}
        \big\|J_{\gamma^{-1}H^{\alpha}}&(-M^*(\gamma^{-1}x+d_H))-J_{\gamma^{-1}H^{\alpha}}(-M^*(\gamma^{-1}y+d_H))\big\| \\
        &=
            \Big\|\frac1{2\gamma}M^{-1}\left( (1+\alpha)(J_{\gamma D_H}x-J_{\gamma D_H}y)-(1+\alpha)(x-y)\right) \Big\|\\
        &\leq
            \frac1{2\gamma\sigma_{\min}(M)}\left( (1+\alpha)\|J_{\gamma D_H}x-J_{\gamma D_H}y\|+(1+\alpha)\|x-y\|\right) \\
        &\leq
            \frac1{2\gamma\sigma_{\min}(M)}\left( (1+\alpha)\|x-y\|+(1+\alpha)\|x-y\|\right)
        =
            \frac{1+\alpha}{\gamma\sigma_{\min}(M)}\|x-y\|,
    \end{align*}
    which shows the result.
\end{proof}

\begin{theorem}[Convergence of the HOST--induced primal algorithm]\label{thm:primal_convergence}
    Let $A,B$ be injective and $A^{-1},B^{-1}$ denote their single--valued preimage maps.
    Let $D_F,D_G\in\NR$.
    be such that there exists, at each iteration $k$, $(F^{\varphi^k})^{-1}$ and $(G^{\theta^k})^{-1}$ such that $R_{\gamma D_{F^{\varphi^k}}}=R^{\varphi^k}_{\gamma D_F}$ and $R_{\gamma D_{G^{\theta^k}}}=R^{\theta^k}_{\gamma D_G}$, respectively.
    Then, as $y^k\rightarrow \bar{y}$ as $k \to \infty$, the shadow and primal variable sequences satisfy the following limits.
    \begin{enumerate}[label=(\roman*)]
        \item\label{item:primal_x} The sequence $(x^k)$ such that $x^{k+1}\coloneqq J^A_{\rho^{-1}F^{\varphi^{k+1}}}(-A^*(Bz^{k}+\rho^{-1}\lambda^{k}+d_{F^{\varphi^{k+1}}}+d_{G^{\varphi^{k+1}}}))$ satisfies
        \begin{equation}
        x^k \rightarrow \bar{x}
        =
        J_{\rho^{-1} F^{{\varphi^{k+1}}}}^A (-A^* (\rho^{-1} R_{\rho D_{G^{\theta^k}}}\bar{y} + d_{F^{\varphi^{k+1}}})). \label{x_limit_resolvent_version}
        \end{equation}
        \item\label{item:primal_z} The sequence $(z^k)$ such that  $
        z^{k+1}\coloneqq J_{\rho^{-1} G^{\theta^k}}^B (-B^* (Ax^k+\rho^{-1} \lambda^{k-1} +d_F+ d_G))$ satisfies
        \begin{equation}
        z^k \rightarrow \bar{z}
        =
        J_{\rho^{-1} G^{\bar{\theta}}}^B (-B^* (\rho^{-1} \bar{y} + d_G)).\label{z_limit_resolvent_version}
        \end{equation}
        \item\label{item:primal_lambda} The sequence $(\lambda^k)$ such that $\lambda^{k+1} \coloneqq \lambda^k+\rho(Ax^{k+1}+Bz^{k+1}+d_F+d_G)$ satisfies $\lambda^k \rightarrow \bar{\lambda} = J_{\rho D_{G^{\bar{\theta}}}}\bar{y}$.
    \end{enumerate}
\end{theorem}
\begin{proof}
    In the case when $(\bar{\varphi},\bar{\theta}) = (\varphi^0,\theta^0)$, the results follow from, for example, \cite[Theorem~15.4]{beck2017methods}.
    We will prove the result for the case $(\bar{\varphi},\bar{\theta}) \neq (\varphi^0,\theta^0)$.
    
    \ref{item:primal_x}.
        By Theorem~\ref{thm:GMI_duality} and Table~\ref{tab:gabay_identities} we have $\lambda^k+\rho(Bz^k+d_G)=R_{\rho D_{G^{\theta^k}}}(y^k)$.
        It follows that 
        \begin{multline*}
            x^{k+1}
            =
            J^A_{\rho^{-1}F^{\varphi^k}}(-A^*(Bz^k+\rho^{-1}\lambda^k+d_F+d_G))
            \\=
            J^A_{\rho^{-1}F^{\varphi^k}}(-A^*(\rho^{-1}R_{\rho D_{G^{\theta^k}}}y^k+d_F))
            =:
            s_{\varphi^k}(y^k).
        \end{multline*}
        The triangle inequality yields
        \begin{multline}
        \big\|x^{k+1}-s_{\bar{\varphi}}(\bar{y})\big\|
        \leq
        \big\|x^{k+1}-s_{\varphi^k}(\bar{y})\big\|+\|s_{\varphi^k}(\bar{y})-s_{\bar{\varphi}}(\bar{y})\|
        \\ \leq
        \big\|x^{k+1}-s_{\varphi^k}(y^k)\big\|+\|s_{\varphi^k}(y^k)-s_{\varphi^k}(\bar{y})\|+\|s_{\varphi^k}(\bar{y})-s_{\bar{\varphi}}(\bar{y})\|.
        \label{x_triangle_ineq}
        \end{multline}
        Here, the first term on the right hand side is zero by definition. By the Lipschitz continuity established in Lemma~\ref{lem:characterisation_of_induced_resolvent}\,\ref{item:J_is_Lipschitz}, the second term is estimated as
        \begin{equation*}
            \|s_{\varphi^k}(y^k)-s_{\varphi^k}(\bar{y})\|
            \leq
            \frac{1+\alpha}{\gamma\sigma_{\min}(M)}\big\|R_{\rho D_{G^{\varphi^k}}}y^k-R_{\rho D_{G^{\varphi^k}}}\bar{y}\big\|
            =
            \frac{(1+\alpha)(1+2\varphi^k)}{\gamma\sigma_{\min}(M)}\|y^k-\bar{y}\|
        \end{equation*}
        and thus converges to zero as $y^k\rightarrow\bar{y}$.
        The third term on the right hand side of \eqref{x_triangle_ineq} converges to zero by Lemma~\ref{lem:characterisation_of_induced_resolvent}\,\ref{item:primal_continuous_in_parameters}.

    \ref{item:primal_z}.
        Substituting the Gabay identity $y^{k}=\lambda^{k-1}+\rho (Ax^{k}+d_F)$, our $z$--update can be expressed as 
        $$
        z^{k+1}=J^B_{\rho^{-1}G^{\theta^k}}(-B^*(\rho^{-1}y^k+d_G))=:s_{\theta^k}(y^k).
        $$
        Similarly applying the triangle inequality as above gives
        \begin{equation*}
        \big\|z^{k+1}-s_{\bar{\theta}}(\bar{y})\big\|\leq \big\|z^{k+1}-s_{\theta^k}(y^k)\big\|+\|s_{\theta^k}(y^k)-s_{\theta^k}(\bar{y})\|+\|s_{\theta^k}(\bar{y})-s_{\bar{\theta}}(\bar{y})\|,
        \end{equation*}
        whose first term on the right hand side is zero by definition, with the second term converging to zero by the Lipschitz continuity of $s_{\theta^k}$ (Lemma~\ref{lem:characterisation_of_induced_resolvent}\,\ref{item:J_is_Lipschitz}) and the fact that $y^k\rightarrow \bar{y}$.
        The third tends to zero by the continuity of $s_{\theta^k}$ in $\theta^k$ (Lemma~\ref{lem:characterisation_of_induced_resolvent}\,\ref{item:primal_continuous_in_parameters}).
    
    \ref{item:primal_lambda}.
        By Table~\ref{tab:gabay_identities} and the single--valuedness of $J_{\rho D_{G^{\theta^k}}}$ we have $\lambda^{k+1}=J_{\rho D_{G^{\theta^k}}}y^{k+1}$.
        It follows that 
        \begin{multline*}
            \|\lambda^{k+1} - J_{\rho D_{G^{\bar{\theta}}}}\bar{y} \|
            \leq
            \|\lambda^{k+1} - J_{\rho D_{G^{\theta^k}}}y^{k+1}\|
            \\+ \| J_{\rho D_{G^{\theta^k}}}y^{k+1} - J_{\rho D_{G^{\theta^k}}}\bar{y}\|
            + \| J_{\rho D_{G^{\theta^k}}}\bar{y} - J_{\rho D_{G^{\bar{\theta}}}}\bar{y}\| .
        \end{multline*}
        The first term is zero by definition.
        By \ref{lem:characterisation_of_induced_resolvent}\,\ref{item:nonexpansive_homotopy_resolvent} and Lemma~\ref{lem:homotopy_alg_convergence}, we have
        $$
        \big\| J_{\rho D_{G^{\theta^k}}}y^k - J_{\rho D_{G^{\theta^k}}}\bar{y}\big\|\leq \|y^k-\bar{y}\|\longrightarrow 0.
        $$
        The third term tends to zero by the continuity of $\smash{J_{\gamma D_{G^{\theta^k}}}}$ with respect to $\theta^k$; see Lemma~\ref{lem:characterisation_of_induced_resolvent}\,\ref{item:continuous_homotopy_resolvent}.
\end{proof}

In view of Theorem~\ref{thm:primal_convergence}, the solution $(\bar{x},\bar{z},\bar{\lambda})$ obtained by HOST solves 
\begin{align*}
    0\in
    \begin{bmatrix}
    F^{\bar{\varphi}}(\bar{x})+A^*\bar{\lambda}\\
    G^{\bar{\theta}}(\bar{z})+B^*\bar{\lambda}\\
    A\bar{x}+B\bar{z}+d_F+d_G
    \end{bmatrix}\quad \text{and} \quad
    & 0\in D_{F^{\bar{\varphi}}}(\lambda)+G_{G^{\bar{\theta}}}(\lambda).
\end{align*}
It is clear that the obtained solution $\bar{\lambda}$ is feasible with respect to the constraint $Ax+Bz+d_F+d_G = 0$. Moreover, in the case where $(\bar{\varphi},\bar{\theta})=(1,1)$, we have that $(\bar{x},\bar{z},\bar{\lambda})$ solves the original problem \eqref{primal_inclusion_problem}. This was the case in all of our experiments. In the next section we will see that for many problems of interest the operators $F_{\varphi},G_{\theta}$ enjoy natural interpretations as homotopies related to $F$ and $G$. 

\section{Applications to nonconvex optimization}\label{sec:experiment_results}

We now consider the application of HOST to nonconvex problems.
We first revisit the nonconvex--regularized basis pursuit instances of section~\ref{ex:basis_pursuit}.
Then we consider \eqref{rlad_problem} in section~\ref{subsec:rlad}.
The dual inclusion of each of these problems is such that $D_F$ is maximal monotone, so one could therefore set $\varphi^k=1$ for all $k\in\N$.
Since $D_G$ is a sparsity--promoting operator, we will see that even if $\theta^k\nrightarrow 1$, our solution still solves a sparse regularized problem where the regularizer is an intermediate map between the nonconvex regularizer we have specified and a convex regularizer to which it is homotopic.
Hence, even if $\theta^k \nrightarrow 1$, our solution is both interpretable and practical for the purpose we intend.

\subsection{Basis pursuit}\label{subsec:basis_pursuit}

The HOST and HOST--induced primal algorithm iterations when HOST is applied to \eqref{problem:basis_pursuit_dual} are visualized in Figure~\ref{fig:basis_pursuit_successful_convergence}.
For these experiments we set $\varphi^k=\theta^k=\smash{\frac{\log(k+1)}{1+\log(k+1)}}$ and $p^k=200/k^{1.1}$.
We maintain the same color scheme as in Figure~\ref{fig:basis_pursuit_failed_convergence}, with HOST iterates shown transitioning from cyan to magenta as $k\rightarrow \infty$ and primal algorithm sequence $(x^k)$ represented using green squares.
Observe that unlike DR, HOST does not exhibit periodic or chaotic behaviour and it resists the attraction of fixed points that do not correspond to useful primal solutions. 

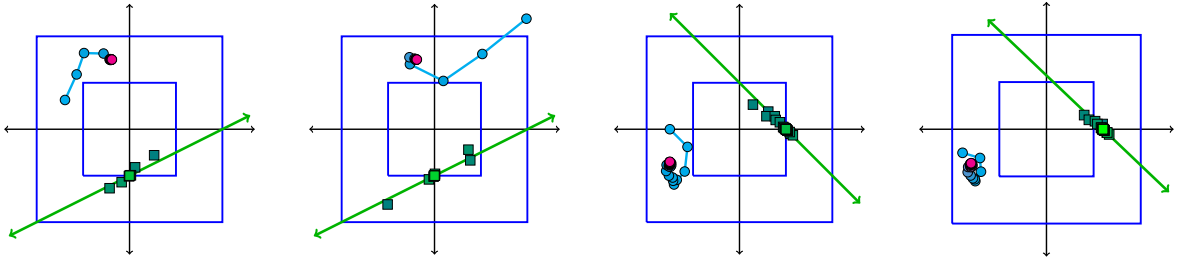
\begin{figure}[htb!]
    \centering%
    \begin{subfigure}[c]{0.24\textwidth}
\centering
\resizebox{0.9\textwidth}{!}{
\begin{tikzpicture}
\begin{axis}[
  axis lines=middle,
  axis line style = {<->, line width=1.2pt},
  xmin=-2.7, xmax=2.7,
  ymin=-2.7, ymax=2.7,
  xlabel={}, ylabel={},
  xticklabels={}, yticklabels={},
  width=9cm, height=9cm,
  axis equal image,
]

\addplot[blue, line width=1.6pt] coordinates {
  (-2,-2) (2,-2) (2,2) (-2,2) (-2,-2)
};

\addplot[blue, line width=1.6pt] coordinates {
  (-1,-1) (1,-1) (1,1) (-1,1) (-1,-1)
};

\addplot[green!70!black, line width=2.2pt, domain=-2.6:2.6,<->] {-1 + 0.5*x};

\addplot[
  scatter,
  scatter src=explicit,
  scatter/use mapped color={draw=black, fill=mapped color},
  only marks,
  mark=*,
  mark size=4pt,
  colormap name = cyan2magenta,
]
table[
  col sep=comma,
  x=x,
  y=y,
  meta expr=\coordindex
]{SuccessfulConvergence_HOST_points.csv};

\addplot[
  scatter,
  scatter src=explicit,
  scatter/use mapped color={draw=black, fill=mapped color},
  only marks,
  mark=square*,
  mark size=4pt,
  colormap name = teal2green,
]
table[
  col sep=comma,
  x=x,
  y=y,
  meta expr=\coordindex
]{SuccessfulConvergence_HADMM_points.csv};

\addplot[
  scatter,
  scatter src=explicit,
  no marks,
  line width = 2pt,
  draw=cyan,
]
table[
  col sep=comma,
  x=x,
  y=y,
  meta expr=\coordindex
]{SuccessfulConvergence_HOST_points.csv};
\end{axis}
\end{tikzpicture}
}
\end{subfigure}
\hfill
\begin{subfigure}[c]{0.24\textwidth}
\centering
\resizebox{0.9\textwidth}{!}{
\begin{tikzpicture}
\begin{axis}[
  axis lines=middle,
  axis line style = {<->, line width=1.2pt},
  xmin=-2.7, xmax=2.7,
  ymin=-2.7, ymax=2.7,
  xlabel={}, ylabel={},
  xticklabels={}, yticklabels={},
  width=9cm, height=9cm,
  axis equal image,
]

\addplot[blue, line width=1.6pt] coordinates {
  (-2,-2) (2,-2) (2,2) (-2,2) (-2,-2)
};

\addplot[blue, line width=1.6pt] coordinates {
  (-1,-1) (1,-1) (1,1) (-1,1) (-1,-1)
};

\addplot[green!70!black, line width=2.2pt, domain=-2.6:2.6,<->] {-1 + 0.5*x};

\addplot[
  scatter,
  scatter src=explicit,
  scatter/use mapped color={draw=black, fill=mapped color},
  only marks,
  mark=*,
  mark size=4pt,
  colormap name = cyan2magenta,
]
table[
  col sep=comma,
  x=x,
  y=y,
  meta expr=\coordindex
]{bad_fixed_success_HOST_points.csv};

\addplot[
  scatter,
  scatter src=explicit,
  scatter/use mapped color={draw=black, fill=mapped color},
  only marks,
  mark=square*,
  mark size=4pt,
  colormap name = teal2green,
]
table[
  col sep=comma,
  x=x,
  y=y,
  meta expr=\coordindex
]{bad_fixed_success_HADMM_points.csv};

\addplot[
  scatter,
  scatter src=explicit,
  no marks,
  line width = 2pt,
  draw=cyan,
]
table[
  col sep=comma,
  x=x,
  y=y,
  meta expr=\coordindex
]{bad_fixed_success_HOST_points.csv};
\end{axis}
\end{tikzpicture}
}
\end{subfigure}
\hfill
\begin{subfigure}[c]{0.24\textwidth}
\centering
\resizebox{0.9\textwidth}{!}{
\begin{tikzpicture}
\begin{axis}[
  axis lines=middle,
  axis line style = {<->, line width=1.2pt},
  xmin=-2.7, xmax=2.7,
  ymin=-2.7, ymax=2.7,
  xlabel={}, ylabel={},
  xticklabels={}, yticklabels={},
  width=9cm, height=9cm,
  axis equal image,
]

\addplot[blue,line width=1.6pt] coordinates {
  (-2,-2) (2,-2) (2,2) (-2,2) (-2,-2)
};

\addplot[blue, line width=1.6pt] coordinates {
  (-1,-1) (1,-1) (1,1) (-1,1) (-1,-1)
};

\addplot[green!70!black, line width=2.2pt, domain=-1.5:2.6,<->] {1 - x};

\addplot[
  scatter,
  scatter src=explicit,
  scatter/use mapped color={draw=black, fill=mapped color},
  only marks,
  mark=square*,
  mark size=4pt,
  colormap name=teal2green,
]
table[
  col sep=comma,
  x=x,
  y=y,
  meta expr=\coordindex
]{periodic_success_HADMM_points.csv};

\addplot[
  scatter,
  scatter src=explicit,
  scatter/use mapped color={draw=black, fill=mapped color},
  only marks,
  mark=*,
  mark size=4pt,
  colormap name = cyan2magenta,
]
table[
  col sep=comma,
  x=x,
  y=y,
  meta expr=\coordindex
]{periodic_success_HOST_points.csv};

\addplot[
  scatter,
  scatter src=explicit,
  scatter/use mapped color={draw=black, fill=mapped color},
  no marks,
  line width = 2pt,
  draw=cyan,
]
table[
  col sep=comma,
  x=x,
  y=y,
  meta expr=\coordindex
]{periodic_success_HOST_points.csv};

\end{axis}
\end{tikzpicture}
}
\end{subfigure}
\hfill
\begin{subfigure}[c]{0.24\textwidth}
\centering
\resizebox{0.9\textwidth}{!}{
\begin{tikzpicture}
\begin{axis}[
  axis lines=middle,
  axis line style = {<->, line width=1.2pt},
  xmin=-2.7, xmax=2.7,
  ymin=-2.7, ymax=2.7,
  xlabel={}, ylabel={},
  xticklabels={}, yticklabels={},
  width=9cm, height=9cm,
  axis equal image,
]

\addplot[blue, line width=1.6pt] coordinates {
  (-2,-2) (2,-2) (2,2) (-2,2) (-2,-2)
};

\addplot[blue, line width=1.6pt] coordinates {
  (-1,-1) (1,-1) (1,1) (-1,1) (-1,-1)
};

\addplot[green!70!black, line width=2.2pt, domain=-1.25:2.6,<->] {1.14 - 0.951954*x};

\addplot[
  scatter,
  scatter src=explicit,
  scatter/use mapped color={draw=black, fill=mapped color},
  only marks,
  mark=*,
  mark size=4pt,
  colormap name = cyan2magenta,
]
table[
  col sep=comma,
  x=x,
  y=y,
  meta expr=\coordindex
]{chaos_success_HOST_points.csv};

\addplot[
  scatter,
  scatter src=explicit,
  scatter/use mapped color={draw=black, fill=mapped color},
  only marks,
  mark=square*,
  mark size=4pt,
  colormap name = teal2green,
]
table[
  col sep=comma,
  x=x,
  y=y,
  meta expr=\coordindex
]{chaos_success_HADMM_points.csv};

\addplot[
  scatter,
  scatter src=explicit,
  no marks,
  line width=2pt,
  draw=cyan,
]
table[
  col sep=comma,
  x=x,
  y=y,
  meta expr=\coordindex
]{chaos_success_HOST_points.csv};

\end{axis}
\end{tikzpicture}}
\end{subfigure}
    \caption{HOST and induced primal sequences generated for instances of \eqref{problem:basis_pursuit_dual}.}%
    \label{fig:basis_pursuit_successful_convergence}%
\end{figure}

Applying Lemma~\ref{lem:over_relaxations_H_theta} with $D=(\partial\reg)^{-1}$, we see that there exists $(\phi_{\mcp}^{\theta})^{-1}$ such that $R_{\gamma(\partial \phi_{\mcp}^{\theta})^{-1}}=R^{\theta}_{\gamma(\partial\phi_{\mcp})^{-1}}$.
Solving for $\phi_{\mcp}^{\theta}$ yields
$$
    \phi^{\theta}_{\mcp}(y)
    =
    \begin{cases}
        \frac{y^2(\beta(1-\theta)-2)}{2\beta(1+\theta)}+\lambda|y|\ifc |y|<\frac{\beta\lambda(1+\theta)}2\\
        \frac{y^2(1-\theta)}{2(1+\theta)}+\frac{\beta\lambda^2(1+\theta)}4\ifc |y|\geq \frac{\beta\lambda(1+\theta)}2.
    \end{cases}
$$
It follows from Theorem~\ref{thm:primal_convergence} that if the HOST shadow sequence $(\lambda^k)$ converges to a solution of \eqref{problem:basis_pursuit_dual}, then the primal sequence it induces $(x^k,z^k)$ converges to a stationary point of 
\begin{equation}
    \minimize_{x,z\in \R^n}\quad \indicator_{\M}^{\varphi}(x)+\sum_{i=1}^n\phi_{\mcp}^{\theta}(z_i)
    \quad\subjectto\quad
    x-z=0.
    \label{induced_basis_pursuit_primal_problem}
\end{equation}

\subsection{Regularized LAD}\label{subsec:rlad}

We now consider the application of HOST to \eqref{rlad_problem}.
We utilize block variables to ensure that the HOST updates are computable in closed form.
This choice is justified in Remark~\ref{rem:blocking}.
Let $z=(\hat{z},\tilde{z})$ be the vertical concatenation of $\hat{z}\in\R^m$ and $\tilde{z}\in\R^n$ and consider the following problem equivalent to \eqref{rlad_problem}.

\begin{equation}
    \minimize_{(x,z)\in\R^n\times\R^{m+n}}\quad \|\hat{z}\|_1+\reg(\tilde{z})
    \quad\subjectto\quad 
    \begin{bmatrix}
    Ux-\hat{z}-w\\
    x-\tilde{z}
    \end{bmatrix}
    =0.
    \label{problem:block_rlad_minimisation}\tag{$\ensuremath{\mathcal{P}}$-$\reg$LAD}
\end{equation}
Setting $A= \big[ U^{\top} \quad I_n \big]^{\top}$, $B=-I_{m+n}$, $d_F=[ -w^{\top} \quad 0 ]^{\top}$ and $d_G=0$, the constraints of \eqref{problem:block_rlad_minimisation} can the expressed in an affine form.
Notice also that the objective of \eqref{problem:block_rlad_minimisation} matches \eqref{primal_optimisation_problem} with $f\equiv0$ and $g(z)=\|\hat{z}\|_1+\reg(\tilde{z})$.

\begin{remark}[Limitation of a natural splitting]\label{rem:blocking}
    A natural way to recast \eqref{rlad_problem} into form \eqref{primal_optimisation_problem} is
    \begin{equation}
        \minimize_{(x,z)\in\R^n\times\R^m}\quad \|z\|_1+ \reg(x)
        \quad\subjectto\quad
        Ux-z-w=0.
        \label{rlad_bad_formulation}
    \end{equation}
    Here the dual operator associated with the $x$ variable is $D_F=-U\circ(\partial\reg)^{-1}\circ-U^{\top}$. By \eqref{GMI} we have 
    $$
    J_{\gamma D_F}(u)=u-\gamma J^U_{\gamma^{-1}\partial\reg}(\gamma^{-1}U^{\top}u).
    $$
    The operator $J^U_{\gamma^{-1}\partial\reg}=(U^{\top}U+\partial\reg)^{-1}$ is not computable in closed form in general, whereas formulation \eqref{problem:block_rlad_minimisation} yields fully explicit updates, as we will see. We thus choose to apply HOST to the dual inclusion associated with \eqref{problem:block_rlad_minimisation} over that of \eqref{rlad_bad_formulation} to reduce the per--iteration computational cost of implementing the method.
\end{remark}

The dual inclusion corresponding to \eqref{problem:block_rlad_minimisation} is to find $\lambda\in\H$ such that 
\begin{equation}
    0\in \underbrace{ -A(F^{-1}(-A^{\top} (\lambda))-d_F}_{=D_F(\lambda)}+\underbrace{G ^{-1}(\lambda)}_{=D_G(\lambda)}.
    \label{LAD_dual_problem}\tag{$\ensuremath{\mathcal{D}}$-$\reg$LAD}
\end{equation}
Here $F^{-1}$ represents the preimage map of the zero function i.e., $F^{-1}(y)=\R^n$ if $y=0$ and $F^{-1}(y)=\emptyset$ otherwise.
Since $g$ is separable, the Clarke subdifferential decomposes componentwise: $\partial g(z)= \partial\|\hat{z}\|_1\times \partial \reg(\tilde{z})=:G(z)$.
The resolvent of $D_F$, computed using Theorem~\ref{thm:Moreaus_identity}, is given by
\begin{multline*}
    J_{\rho D_F}(y)
    =
    y+\rho d_F+\rho A\left( A^{\top}A+\rho^{-1}F\right) ^{-1}(-A^{\top}(\rho^{-1}y+ d_F))
    \\=
    y+\rho d_F-\rho A( A^{\top}A)^{-1}A^{\top}(\rho^{-1}y+ d_F).
\end{multline*}
Similar application of \eqref{GMI} yields
\begin{equation}
    J_{\rho D_G}(z)
    =
    z-\rho\left( I+\rho^{-1}G\right) ^{-1}(\rho^{-1}z)
    =
    z-(\rho J_{\rho^{-1}(\partial \|\cdot\|_1)}(\rho^{-1}\hat{z}),\rho S_{\rho^{-1}(\partial \reg)}(\rho^{-1}\tilde{z})).
    \label{equ:LAD_D1_resolvent}
\end{equation}
The mapping $J_{\rho^{-1}(\partial\||\cdot\|_1)}$ corresponds to the so-called soft--thresholding operator: $J_{\gamma(\partial\||\cdot\|_1)}(x)_i=\operatorname{sign}(x_i)\max\{0,|x_i|-\gamma\}$.
By Proposition~\ref{prop:prox_is_selection}, selections of the proximal operators of MCP and SCAD may serve as selection $S_{\rho^{-1}(\partial \reg)}$.
We refer the reader to \href{https://proximity-operator.net/}{proximity-operator.net} for expressions of these mappings.

\subsubsection{Numerical Results}

We numerically validate HOST by comparing the solutions it locates for \eqref{rlad_problem} with MCP and SCAD regularization to those obtained by CVXPY \cite{boyd2016cvxpy} with $\reg=\lambda\|\cdot\|_1$.
We refer to MCP and SCAD regularized \eqref{rlad_problem} as MCP-LAD and SCAD-LAD, respectively.
The methods are compared on eight problem instances, with the design and response data taken from the Faraway and scikit--learn datasets \cite{faraway2020data,pedregosa2011scikitlearn}.
The data is standardised by centering each feature to have zero mean and scaling to unit variance.
Each dataset is randomly sampled to form model training sets $U_{\mathrm{train}}$ and $w_{\mathrm{train}}$, where each algorithm sees the same training set in each instance.
The remaining samples $U_{\mathrm{test}}$ and $w_{\mathrm{test}}$ are used to validate each model's accuracy.

Appropriate regularizer weights for each solver are determined by running each instance over 50 values in a logarithmically--spaced grid of weights $\lambda\in[0.1,10]$.
We then compare the candidate solution $x^*$ returned by each solver under the weight that yielded the smallest average test error: $\frac1n\|U_{\mathrm{test}}x^*-w_{\mathrm{test}}\|_1$.
 
We applied HOST to MCP-LAD and SCAD-LAD with regularizer parameters $\beta$ and $a$ selected empirically for each problem instance.
The parameters $\rho=1$ and $k_{\max}=2000$ are fixed for all HOST experiments.
Since $D_F$ in \eqref{LAD_dual_problem} is maximal monotone, we fix $\varphi^k=1$ for all $k$.
We define the sequence $(\theta^k)$ according to the affine ramp function:
$$
\theta^k=
\begin{cases}
0\ifc k\leq 100\\
\frac{k-100}{700} \theta_{\max} \ifc k\in (100,800)\\
\theta_{\max}\ifc k\geq 800.
\end{cases}
$$

We set $\theta_{\max}=1$ for all HOST experiments. We also include the results of HOST$^{\dagger}$, where we set $\theta_{\max}=0.8$ to observe performance in cases where $\theta^k\nrightarrow 1$. The average test error and \emph{sparsity} of each $x^*$ i.e., the number of indices $i$ such that $|x_i^*|>0.1$, are reported in Table~\ref{table:experiments}. For reproducibility, the chosen parameter values are detailed in accompanying code.

{\setlength{\tabcolsep}{2.5pt} 
\begin{table}[tbh]
    \centering
    \small 
    \begin{tabular}{c | cc | cc | cc | cc}
        \textbf{Problem}&\multicolumn{2}{c|}{$\boldsymbol{\ell}_{\mathbf{1}}$\bfseries -- CVXPY}&\multicolumn{2}{c|}{\bfseries MCP -- HOST$^{\boldsymbol{\dagger}}$}&\multicolumn{2}{c|}{\bfseries MCP -- HOST}  & \multicolumn{2}{c}{\bfseries SCAD -- HOST}\\
        \textbf{instance} & sparsity&ave. error &sparsity&ave. error&sparsity&ave. error&sparsity & ave. error \\
        \hline
        1&1&0.0488&2&\emph{0.1383}&1&0.0484&1&0.0489\\
        2&3&0.6114&3&0.6152&3&0.6098&3&0.6103\\
        3&4&0.4396&4&0.4426&4&0.4399&4&0.4397\\
        4&9&0.526&12&\textbf{0.5085}&9&\textbf{0.4978}&9&0.5259\\
        5&5&0.2401&3&0.2407&4&\textbf{0.2315}&5&\textbf{0.2345}\\
        6&7&0.5371&6&\textbf{0.5307}&6&0.5331&6&0.5319\\
        7&5&0.6081&5&0.612&5&0.6075&5&0.6075\\
        8&1&0.5313&1&0.5317&1&0.5312&1&0.5313
    \end{tabular}
    \caption{Benchmarking the solution quality of HOST for MCP and SCAD regularized \eqref{rlad_problem} against CVXPY applied to \eqref{rlad_problem} with $\ell_1$ regularization.}
    \label{table:experiments}
\end{table}}

We observe in Table~\ref{table:experiments} that HOST applied to MCP-LAD or SCAD-LAD, in general, returns solutions with a small number of nonzero entries and error than CVXPY on \eqref{rlad_problem} with $\ell_1$ regularization.
Errors are bolded if they are less than 1\% smaller than the error for $\ell_1$-LAD.
Errors are italicized if they are more than 1\% higher than the $\ell_1$-LAD error. 

DR is a natural heuristic for \eqref{rlad_problem} because of the nonsmoothness of both terms. We compare the solutions obtained by DR and HOST for SCAD-LAD with problem instance 8 in Figure~\ref{fig:DRvsHOST}.
The plot shows the objective value of candidate solutions obtained by each method over the logarithmically--spaced grid of weights $\lambda\in[0.1,100]$. We also show the average error of these solutions using dashed lines. HOST and DR are shown in blue and red respectively. We set a maximum number of 2000 iterates, although the methods may stop earlier if $\|y^{k+1}-y^k\|\leq 10^{-3}$.

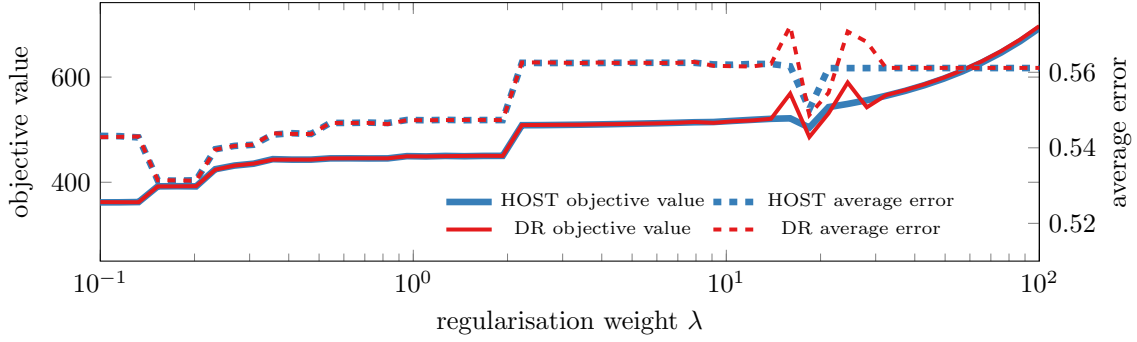
\begin{figure}[tbh]
    \centering
    \begin{tikzpicture}[scale=1]
\begin{axis}[
    width=14cm,
    height=5cm,
    xmode=log,
    xmin=0.1,
    xmax=100,
    xlabel={regularisation weight $\lambda$},
    ylabel={objective value},
    ymin=250,
    xmajorgrids=false,
    ymajorgrids=false,
    legend to name=combinedlegend,
    legend columns = 2,
    legend style={font=\scriptsize,draw=none},
]
\addplot[
    line width = 2.5pt,
    mark={},
    mark size=2pt,
    cb2blue
] table[
    col sep=comma,
    x=lambda,
    y=objective
] {HOST_objective.csv};
\addlegendentry{HOST objective value}

\addlegendimage{cb2blue, dashed, line width=2.5pt}
\addlegendentry{HOST average error}

\addplot[
    line width = 1.5pt,
    mark={},
    mark size=2pt,
    cb2red
] table[
    col sep=comma,
    x=lambda,
    y=objective
] {DR_objective.csv};
\addlegendentry{DR objective value}

\addlegendimage{cb2red, dashed, line width=1.5pt}
\addlegendentry{DR average error}
\end{axis}

\begin{axis}[
    width=14cm,
    height=5cm,
    xmode=log,
    xmin=0.1,
    xmax=100,
    ymin=0.51,
    axis x line=none,
    axis y line*=right,
    xlabel={$\lambda$},
    ylabel={average error},
    xmajorgrids=false,
    ymajorgrids=false,
]
\addplot[
    line width = 2.5pt,
    dashed,
    mark={},
    mark size=2pt,
    cb2blue
] table[
    col sep=comma,
    x=lambda,
    y=error
] {HOST_ave_error.csv};
\addplot[
    line width = 1.5pt,
    dashed,
    mark={},
    mark size=2pt,
    cb2red
] table[
    col sep=comma,
    x=lambda,
    y=error
] {DR_ave_error.csv};
\end{axis}

\node at (7.95,0.6) {\pgfplotslegendfromname{combinedlegend}};
\end{tikzpicture}
    \caption{Objective value and average error of solutions generated by HOST and DR for SCAD--regularized \eqref{rlad_problem} for varying regularization strength parameters.}
    \label{fig:DRvsHOST}
\end{figure}

\section{Conclusion}

This work proposed a general algorithmic duality framework to derive and explore the relationships between resolvent splitting methods.
We use this framework to establish the duality between DR and an ADMM--like method, extending the well known relationship between DR and ADMM in the convex setting.
We additionally show certain relationships between the methods' iterates to hold in general.
This result is fundamental for, for instance, implementing dual--based extrapolation schemes.
We use our framework to show that applying DR to dual problems may fail to generate convergent solution sequences.

Motivated by DR's observed failure for certain problems, we utilised the algorithmic duality framework to propose HOST: a convergent scheme for the class of two--operator inclusion problems where both operators enjoy nonexpansive resolvents.
Revisiting problem instances where DR failed to converge to useful solutions, we show that HOST is naturally more resilient than DR for this problem.
We also use these instances to illustrate HOST's warm--starting and homotopy--inducing nature.

We validate the practical relevance of HOST with numerical experiments for the nonsmooth and nonconvex--regularized least absolute deviations problem.
We compare HOST's solutions to those generated by CVXPY for $\ell_1$--regularized least absolute deviations.
This comparison required the development of a novel formulation for \eqref{rlad_problem} based on block variables.
Our results show that applying HOST to \eqref{rlad_problem} often yields superior solutions to established approaches.

\section*{Acknowledgements}

The MATRIX Institute workshop \emph{Splitting Algorithms - Advances, Challenges, and Opportunities} was instrumental to the success of this work.
Participation of all authors was made possible by generous financial support from MATRIX Institute, Simons Foundation, and Australian Mathematical Sciences Institute, as well as Australian Mathematical Society and its special interest group Mathematics of Computation and Optimisation, and an in-kind contribution from Drs.\ Minh N. Dao and Lien Nguyen.

\section*{Declarations}

\subsection*{Online availability of experimental results}
All code and experimental results associated with this work are publicly available and archived on Zenodo at \textsc{doi}: \href{https://doi.org/10.5281/zenodo.19910577}{10.5281/zenodo.19910577}.

\subsection*{Statement about AI use}
The authors declare that no artificial intelligence tools or large language models were used in the research, writing, or preparation of this manuscript.

\bibliographystyle{plain}
\bibliography{references}

\appendix

\section{MCP and SCAD}\label{app:mcp_scad_derivations}

By Proposition~\ref{prop:prox_is_selection}, the well--defined and efficiently--computable expression for $\smash{\prox_{\tau \phi_{\mcp}}}$ supplies a valid selection of $\smash{J_{\tau \partial\phi_{\mcp}}}$ \cite{zhang2010mcp}.
We may thus obtain a closed--form selectant $\smash{S_{\gamma(\partial\phi_{\mcp})^{-1}}}$ by combining $\smash{\prox_{\tau\phi_{\mcp}}}$ (with $\tau=\gamma ^{-1}$) with Corollary~\ref{selection_GMI_corollary} (with $M=-I$ and $d=0$):
\begin{equation*}
	S_{\gamma(\partial \phi_\mcp)^{-1}}(x)=x-\gamma \prox_{\gamma^{-1}\phi_{\mcp}}(\gamma^{-1}x)=
	\begin{cases}
		x
		\ifc |x|\leq\lambda,\\
		\frac{x-\sgn(x)\beta\gamma\lambda}{1-\beta\gamma}
		\ifc |x|\in({\lambda},\beta\gamma\lambda),\\
		0
		\ifc |x|\geq \beta\gamma\lambda.
	\end{cases}
\end{equation*}
Choosing $\beta$ and $\gamma$ such that $\beta\gamma\geq 2$, one sees that $\smash{S_{\gamma (\partial \phi_{\mcp})^{-1}}}$ is nonexpansive.
Hence, $(\partial\phi_{\mcp})^{-1}\in \NR$ for these parameter choices.

The SCAD is defined as follows for $a>2$ and $\lambda>0$ \cite{fan2001scad}:
\begin{equation*}
	\phi_{\scad}(x)=
	\begin{cases}
		\lambda|x|\ifc |x|\leq\lambda\\
		\frac{2a\lambda-\lambda^2-x^2}{2(a-1)}\ifc |x|\in \left( \lambda,a\lambda \right]\\
		\frac{\lambda^2(a+1)}{2}&\text{otherwise}.
	\end{cases}
\end{equation*}
Using a similar procedure to that above, we find that $S_{\gamma(\partial\phi_{\scad})^{-1}}$ is given by
\begin{equation}
	S_{\gamma(\partial\phi_{\scad})^{-1}}(x)=
	\begin{cases}
		0\ifc |x|\geq a\gamma\lambda\\
		\frac{\sgn(x)a\gamma\lambda-x}{\gamma(a-1)-1}\ifc |x|\in\left( \lambda(\gamma+1),a\gamma\lambda\right) \\
		\sgn(x)\lambda\ifc |x|\in \left( \lambda,\lambda(\gamma+1) \right]\\
		x&\text{otherwise.}
	\end{cases}
	\label{scad_dual_resolvent}
\end{equation}
As long as $a$ and resolvent parameter $\gamma$ are chosen such that $\gamma(a-1)\geq 2$, the selectant $S_{\gamma(\partial\phi_{\scad})^{-1}}$ enjoys nonexpansiveness.

\end{document}